\documentclass[12pt,reqno]{amsart}
\usepackage{geometry}                
\usepackage{color}
\usepackage{graphicx}
\usepackage{subfigure}
\usepackage{amssymb}
\usepackage{amsmath}
\usepackage{multirow} 
\usepackage{enumerate}
\usepackage{epstopdf}
\usepackage{mathrsfs}
\usepackage{epic,eepic}
\usepackage{mathrsfs}
\usepackage{amsthm}

\begin{document}
\vfuzz2pt 
\hfuzz2pt 
\newtheorem{thm}{Theorem}[section]
\newtheorem{model}{Model}
\newtheorem{pro}{Problem}[section]
\newtheorem{cor}[thm]{Corollary}
\newtheorem{lem}[]{Lemma}[section]
\newtheorem{prop}[]{Proposition}[section]
\theoremstyle{definition}
\newtheorem{defn}[thm]{Definition}
\theoremstyle{remark}
\newtheorem{rem}[]{Remark}[section]
\numberwithin{equation}{section}
\newtheorem{col}{Conclusion}

\baselineskip 17pt

\title[]{Self-similar solutions of the spherically symmetric Euler equations for general equations of state}%
\author[]{Jianjun Chen$^\dag$ and Geng Lai$^\ddag$}%
\address{}%
\email{}%

\thanks{$^\ddag$Corresponding author. E-mail: mathchenjianjun@163.com(Chen), laigeng@shu.edu.cn(Lai)}

\dedicatory{$^{\dag}$Department of Mathematics, Zhejiang University of Science and Technology, Hangzhou, 310023, P.R. China\\
$^{\ddag}$Department of Mathematics, Shanghai University,
Shanghai, 200444, P.R. China}%



\begin{abstract}
The study of spherically symmetric motion is important for the theory of explosion waves.
In this paper, we construct rigorously self-similar  solutions to the Riemann problem of the spherically symmetric Euler equations for general equations of state.
We used the assumption of self-similarity to reduce the spherically symmetric Euler equations to a system of nonlinear ordinary differential equations, from which we obtain detailed structures of solutions besides their existence.


%

\
\vskip 0pt
\noindent%
\vskip 4pt
\noindent%
\end{abstract}

\maketitle
\section{\bf Introduction}

The 3D isentropic Euler equations has the form
\begin{equation}\label{3d}
\left\{
   \begin{aligned}
    &\rho_t+(\rho u_1)_{x_1}+(\rho u_2)_{x_2}+(\rho u_3)_{x_3}=0, \\
   &(\rho u_1)_t+(\rho u_1^2+p)_{x_1}+(\rho u_1u_2)_{x_2}+(\rho u_1 u_3)_{x_3}=0, \\
   &(\rho u_2)_t+(\rho u_1 u_2)_{x_1}+(\rho u_2^2+p)_{x_2}+(\rho u_2 u_3)_{x_3}=0,\\
&(\rho u_3)_t+(\rho u_1 u_3)_{x_1}+(\rho u_2 u_3)_{x_2}+(\rho u_3^2+p)_{x_3}=0,
  \end{aligned}
\right.
\end{equation}
where $\rho$ is the density, $(u_1, u_2, u_3)$ is the velocity, and $p=p(\rho)$ is the pressure.

The global existence of solution to the Cauchy problem for system (\ref{3d}) is still a complicated open problem. Thus it has been profitable to consider some special problems.
In this paper, we consider system (\ref{3d}) with
the Riemann initial data
\begin{equation}\label{3db}
\big(\rho, u_1, u_2, u_3\big)(0, x_1, x_2, x_3)~=~\big(\rho_0, u_0\sin\varphi\cos\theta, u_0\sin\varphi\sin\theta, u_0\cos\varphi\big),
\end{equation}
where $(x_1, x_2, x_3)=(r\sin\varphi\cos\theta, r\sin\varphi\sin\theta, r\cos\varphi)$, $r>0$ is the radial variable, $\varphi\in [0, \pi]$, $\theta\in [0, 2\pi)$, and $\rho_0$ and $u_0$ are two constants.

The problem (\ref{3d}), (\ref{3db})
allows us to look for spherically symmetric solution, i.e., $$\rho=\rho(t, r),\quad  u_1=u(x, t)\sin\varphi\cos\theta, \quad u_2=u(x, t)\sin\varphi\sin\theta,\quad u_3=u(x, t)\cos\varphi.$$
We can then reduce system (\ref{3d}) to
\begin{equation}
\left\{
          \begin{aligned}
            &\rho_t+(\rho u)_x+\frac{2\rho u}{x}=0,\\
&(\rho u)_x +(\rho u^2+p)_x+\frac{2\rho u^2}{x}
=0.
            \end{aligned}
       \right.                       \label{AE}
\end{equation}
Then the problem (\ref{3d}), (\ref{3db}) can be reduced to a Riemann initial-boundary value problem for (\ref{AE}) with the initial and boundary conditions
\begin{equation}\label{IBV}
(u, \rho)(x, 0)=(u_0, \rho_0), \quad (\rho u)(0, t)=0.
\end{equation}
The problem (\ref{AE}), (\ref{IBV}) allows us to look for self-similar solutions that depend only on the self-similar variable $\xi=x/t$.

The self-similar solutions for (\ref{AE}) was first studied by Guderley, Taylor, et al; see \cite{CF} and the survey paper \cite{Jenssen}.  Taylor \cite{Ta} used the assumption of self-similarity to reduce the spherically symmetric Euler equations for polytropic gases to a system of nonlinear autonomous ordinary differential equations and solved the ``spherical piston" problem. Zhang and Zheng \cite{Zheng} constructed several 2D self-similar radially symmetric solutions with swirl for polytropic gases.  Hu \cite{Hu} constructed 2D self-similar axisymmetric solutions for the Euler equations for a two-constant equation of state.   For more general existence of weak solutions of (\ref{AE}), we refer the reader to \cite{CJ, ChenG, CGM, DL, LW, MMU1, MMU2}.



\begin{figure}[htbp]
\begin{center}
\includegraphics[scale=0.26]{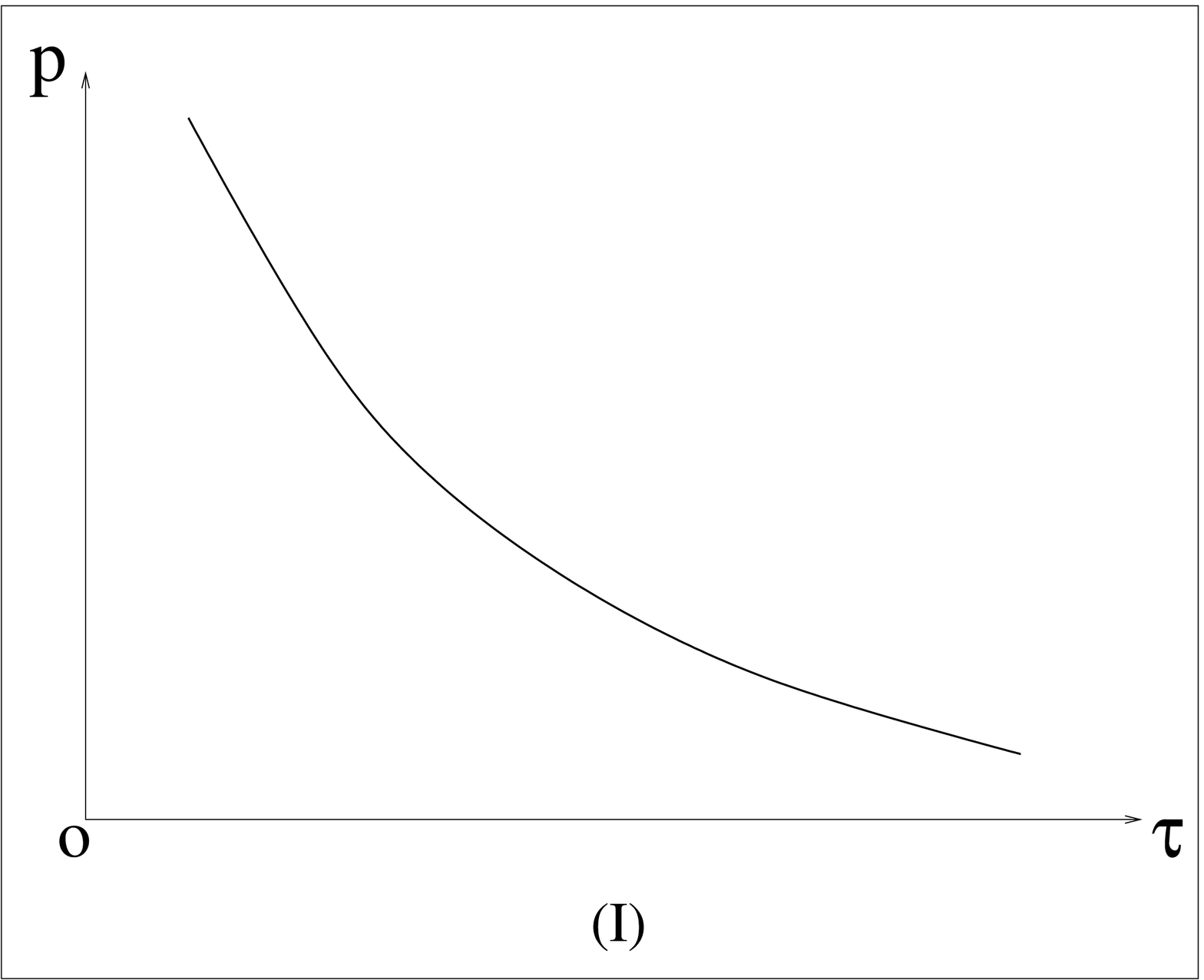}~~\includegraphics[scale=0.26]{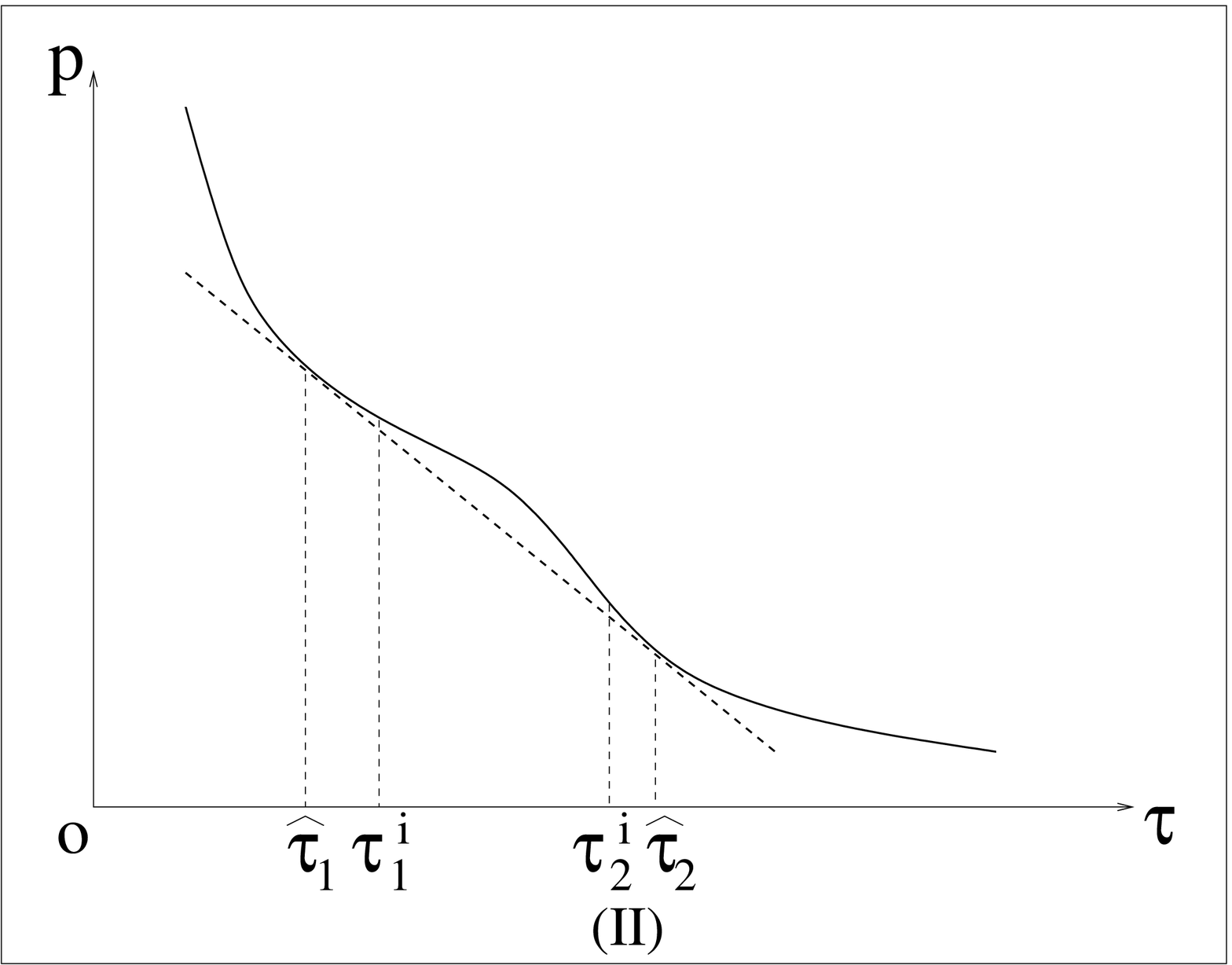}
~~\includegraphics[scale=0.26]{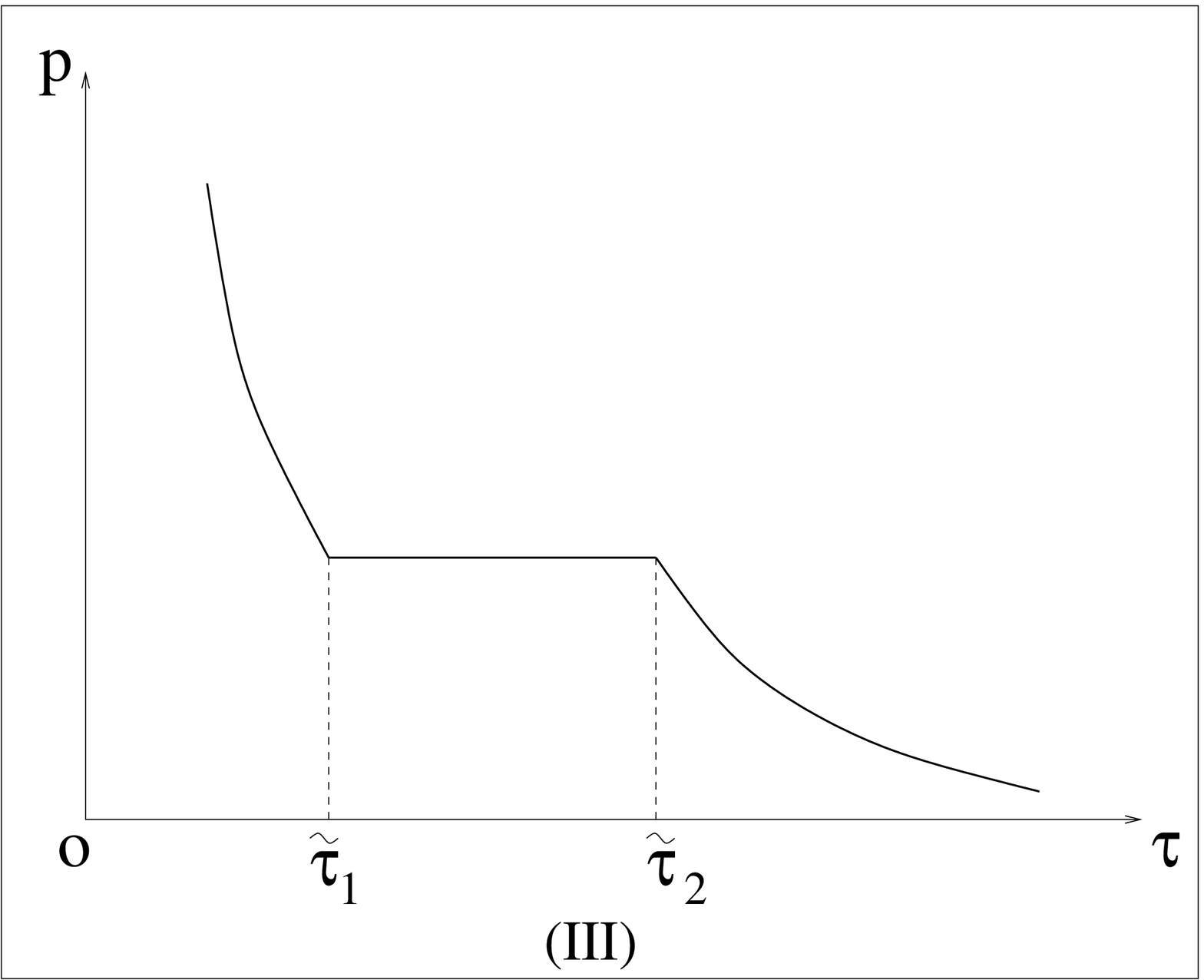}
\caption{\footnotesize Equations of sate.}
\label{Fignew1}
\end{center}
\end{figure}

In this paper, we study the problem (\ref{AE}), (\ref{IBV}) for the following three types of equations of state:
\begin{description}
  \item[I] $p'(\tau)<0$ and $p''(\tau)>0$ as $\tau>0$; see Figure \ref{Fignew1}(I).
  \item[II] $p'(\tau)<0$ as $\tau>0$; $p''(\tau)>0$ as $\tau\in (0, \tau_1^i)\cup (\tau_2^i, +\infty)$;  $p''(\tau)<0$ as $\tau\in (\tau_1^i, \tau_2^i)$; see Figure \ref{Fignew1}(II).
  \item[III] $p'(\tau)<0$ and $p''(\tau)>0$ as $\tau\in (0, \tilde{\tau}_1)\cup (\tilde{\tau}_2, +\infty)$; $p(\tau)$ is constant as $\tau\in[\tilde{\tau}_1, \tilde{\tau}_2]$; see Figure \ref{Fignew1}(III).
\end{description}
Here, $\tau=1/\rho$ is the specific volume.
These three types of equations of state can be referred for instance to the van der Waals equation of state
$
p=\frac{A}{(\tau-1)^{\gamma}}-\frac{1}{\tau^2},
$
where $A$ is a constant corresponding to the entropy, $\gamma$ is a constant between $1$ and $5/3$.
The third type equation of state may be seen as a van der Waals equation of state complemented with Maxwell's  equal areas law and may be used as a simple model of phase transition; see \cite{GN, MP} and the references cited therein.

\begin{rem}
For equation of state II, there exist $\hat{\tau}_1<\tau_1^i<\tau_2^i<\hat{\tau}_2$ such that
$$
\frac{p(\hat{\tau}_{1})-p(\hat{\tau}_2)}{\hat{\tau}_{1}-\hat{\tau}_2}=p'(\hat{\tau}_{1})=p'(\hat{\tau}_2).
$$
\end{rem}

We make the following assumptions about these equations of state:
\begin{description}
  \item[(A1)] There exists a $\nu>0$ such that $\lim\limits_{\rho\rightarrow 0}\frac{p'(\rho)}{\rho^{\nu}}=0$.
  \item[(A2)] For equation of state III, we assume $\lim\limits_{\tau\rightarrow \tilde{\tau}_1^{-}}p'(\tau)<p'(\tau_c)$, where $\tau_c>\tilde{\tau}_2$ is determined by
 $\frac{p(\tau_{c})-p(\tilde{\tau}_1)}{\tau_{c}-\tilde{\tau}_1}=p'(\tau_{c})$.
\end{description}


The main result of the paper can be stated as follows.
\begin{thm}
For equations of state I--III, the Riemann initial-boundary value problem (\ref{AE}), (\ref{IBV}) has a solution for any data $(u_0, \rho_0)$.
\end{thm}

We use the assumption of self-similarity to reduce the spherically symmetric Euler equations (\ref{AE}) to a system of nonlinear ordinary equations, from which we obtain detailed structures of solutions of (\ref{AE}), (\ref{IBV}) besides their existence. There are many differences between our results and the previous results for polytropic gases.
First, system (\ref{AE})
cannot by self-similar transformation be reduced to an autonomous system of ordinary differential equations
for general equations of state, so that the method in \cite{CF,Zheng1} can not be used in here.
Second, the solution for (\ref{AE}), (\ref{IBV}) for polytropic gases is continuous as $u_0>0$, whereas the solution for nonconvex equations of state may be discontinuous as $u_0>0$.
Third, the solution for (\ref{AE}), (\ref{IBV}) for polytropic gases contains only one shock as $u_0<0$, whereas the solution for nonconvex equations of state
 may contain two or even more shocks as $u_0<0$.




\section{\bf Preliminaries}

\subsection{Ordinary equations}
  By self-similar transformation, system (\ref{AE}) can be written as
$$
\left\{
 \begin{aligned}
&-\xi\frac{{\rm d} \rho}{{\rm d} \xi}+\frac{{\rm d} (\rho u) }{{\rm d} \xi}+\frac{2\rho u}{\xi}=0,\\
  &-\xi\frac{{\rm d} u}{{\rm d} \xi}+ u\frac{{\rm d} u}{{\rm d} \xi}+\frac{1}{\rho}\frac{{\rm d} p}{{\rm d} \xi}=0.
  \end{aligned}
\right.
$$
Hence,
\begin{equation}\label{ODE1}
\left\{
 \begin{aligned}
   &\frac{{\rm d} u}{{\rm d} \xi}=-\frac{2p'(\rho) u}{\xi
   \big[p'(\rho)-(u-\xi)^2\big]}, \\
    &\frac{{\rm d} \rho}{{\rm d} \xi}=\frac{2\rho u(u-\xi)}{\xi
   \big[p'(\rho)-(u-\xi)^2\big]}, \\
  \end{aligned}
\right.
\end{equation}

Let $s=1/\xi$. Then, system (\ref{ODE1}) can be changed into
\begin{equation}\label{ODE2}
\left\{
  \begin{aligned}
   &\frac{{\rm d} u}{{\rm d} s}=\frac{2p'(\rho) us}{s^2 p'(\rho)-(1-us)^2},  \\
     &\frac{{\rm d} \rho}{{\rm d} s}=\frac{2\rho u(1-us)}{s^2 p'(\rho)-(1-us)^2}.
  \end{aligned}
\right.
\end{equation}
The initial condition $(u, \rho)(x, 0)=(u_0, \rho_0)$ can be changed into
\begin{equation}\label{ID1}
(u, \rho)\mid_{s=0}~=~(u_0, \rho_0).
\end{equation}
The initial value problem (\ref{ODE2}), (\ref{ID1}) is a classically well-posed problem which has a unique local solution for any $(u_0, \rho_0)$.
Throughout the paper, we denote by $(u_1, \rho_1)(s)$ the (local) classical solution of the initial value problem (\ref{ODE2}), (\ref{ID1}).

In view of the denominators of the right parts of (\ref{ODE2}), we define
\begin{equation}\label{h}
h(\rho_1(s), s):=~\frac{1}{s}-\sqrt{p'\big(\rho_1(s)\big)}.
\end{equation}
Then we have the following properties:
\begin{itemize}
  \item if $u_1(s)<h(\rho_1(s), s)$ then $s^2 p'(\rho_1)-(1-u_1s)^2<0$;
  \item if $u_1(s)=h(\rho_1(s), s)$ then $s^2 p'(\rho_1)-(1-u_1s)^2=0$;
  \item if $h(\rho_1(s), s)<u_1(s)<\frac{1}{s}+\sqrt{p'\big(\rho_1(s)\big)}$ then $s^2 p'(\rho_1)-(1-u_1s)^2>0$.
\end{itemize}

\subsection{Shock waves}
It is known that a weak solution $(u, \rho)$ to (\ref{AE}) satisfies the Rankine-Hugoniot condition across any discontinuity at $(x, t)$:
\begin{equation}
\frac{\rho_1 u_1-\rho_2 u_2}{\rho_1-\rho_2}~=~\frac{\rho_1 u_1^2+p_1-\rho_2 u_2^2-p_2}{\rho_1 u_1-\rho_2 u_2}~=~\sigma
\end{equation}
where $(u_1, \rho_1)=(u, r)(x+0, t)$, $(u_2, \rho_2)=(u, r)(x-0, t)$, and $\sigma$ is the speed of discontinuity. For any $(u_*, \rho_*)$, we let the shock set through $(u_*, \rho_*)$
be the set of points $(u, \rho)$ satisfying the Rankine-Hugoniot condition
$$
\frac{\rho_* u_*-\rho u}{\rho_*-\rho}~=~\frac{\rho_* u_*^2+p_*-\rho u^2-p}{\rho_* u_*-\rho u}~=~\sigma(u_*, \rho_*; u, \rho).
$$
We need to use the entropy condition (E) given by Liu \cite{Liu}.
\begin{defn}
A discontinuity between two states $(u_1, \rho_1)$ and $(u_2, \rho_2)$ satisfies the entropy condition (E) if
\begin{equation}\label{EEntropy}
\sigma(u_1, \rho_1; u_2, \rho_2)\geq \sigma(u_1, \rho_1; u, \rho)
\end{equation}
for all $(u, \rho)$ on the shock set through $(u_1, \rho_1)$ between $(u_1, \rho_1)$ and $(u_2, \rho_2)$.
A shock which satisfies the entropy condition (E) will be called an admissible shock.
\end{defn}

In this paper, we are only concerned with forward shock waves.
So, we give a geometric interpretation of entropy condition (E) for forward shock waves.
\begin{lem}\label{lem21}
A forward shock between two states $(u_1, \tau_1)$ and $(u_2, \tau_2)$ satisfies the entropy condition (E) if and only if
\begin{equation}
\sqrt{-\frac{p_2-p_1}{\tau_2-\tau_1}}~\geq~ \sqrt{-\frac{p-p_1}{\tau-\tau_1}}
\end{equation}
for all $\tau\in\big(\min\{\tau_1, \tau_2\}, \max\{\tau_1, \tau_2\}\big)$. Here, ``1" denotes the fluid in front of the shock, ``2" denotes the fluid behind the shock.
\end{lem}
\begin{proof}
From the Rankine-Hugoniot conditions for forward shock waves we have
\begin{equation}\label{RH}
\left\{
  \begin{array}{ll}
    \rho_1(u_1-\sigma)=\rho_2(u_2-\sigma)<0,  \\[4pt]
    \rho_1(u_1-\sigma)^2+p_1=\rho_2(u_2-\sigma)^2+p_2.
  \end{array}
\right.
\end{equation}


From (\ref{RH}) we get
$$
\frac{\rho_2^2(u_2-\sigma)^2}{\rho_1}+p_1=\rho_2(u_2-\sigma)^2+p_2,
$$
and consequently
$$
(u_2-\sigma)^2=\frac{p_2-p_1}{\rho_2-\rho_1}\cdot\frac{\rho_1}{\rho_2}=-\tau_2^2\frac{p_2-p_1}{\tau_2-\tau_1}.
$$
Thus, we have
\begin{equation}\label{1}
\sigma(u_1, \rho_1; u_2, \rho_2)=u_2+\tau_2\sqrt{-\frac{p_2-p_1}{\tau_2-\tau_1}}.
\end{equation}

Similarly, we have
\begin{equation}\label{2}
\sigma(u_1, \rho_1; u_2, \rho_2)=u_1+\tau_1\sqrt{-\frac{p_2-p_1}{\tau_2-\tau_1}}.
\end{equation}
Thus, for all $(u, \rho)$ on the forward shock set through $(u_1, \rho_1)$ we have
\begin{equation}\label{102503}
\sigma(u_1, \rho_1; u, \rho)=u_1+\tau_1\sqrt{-\frac{p-p_1}{\tau-\tau_1}}.
\end{equation}
Then by (\ref{EEntropy}) we get this lemma.
\end{proof}

We define
$$
\phi(\tau; u_1, \tau_1):=u_1+(\tau_1-\tau)\sqrt{-\frac{p-p_1}{\tau-\tau_1}}.
$$
Then by (\ref{1}), (\ref{2}), and Lemma \ref{lem21}, we have the following corollaries about forward admissible shocks.
\begin{cor}\label{cor1}
For equation of state I, the set $\mathcal{S}_{c}$ of the sates which can be connected to $(u_1, \tau_1)$ by a forward admissible compression shock on the left is given by:
$$
\mathcal{S}_{c}=\{(u, \tau)\mid u=\phi(\tau; u_1, \tau_1), \tau<\tau_1\}.
$$
\end{cor}
\begin{figure}[htbp]
\begin{center}
\includegraphics[scale=0.25]{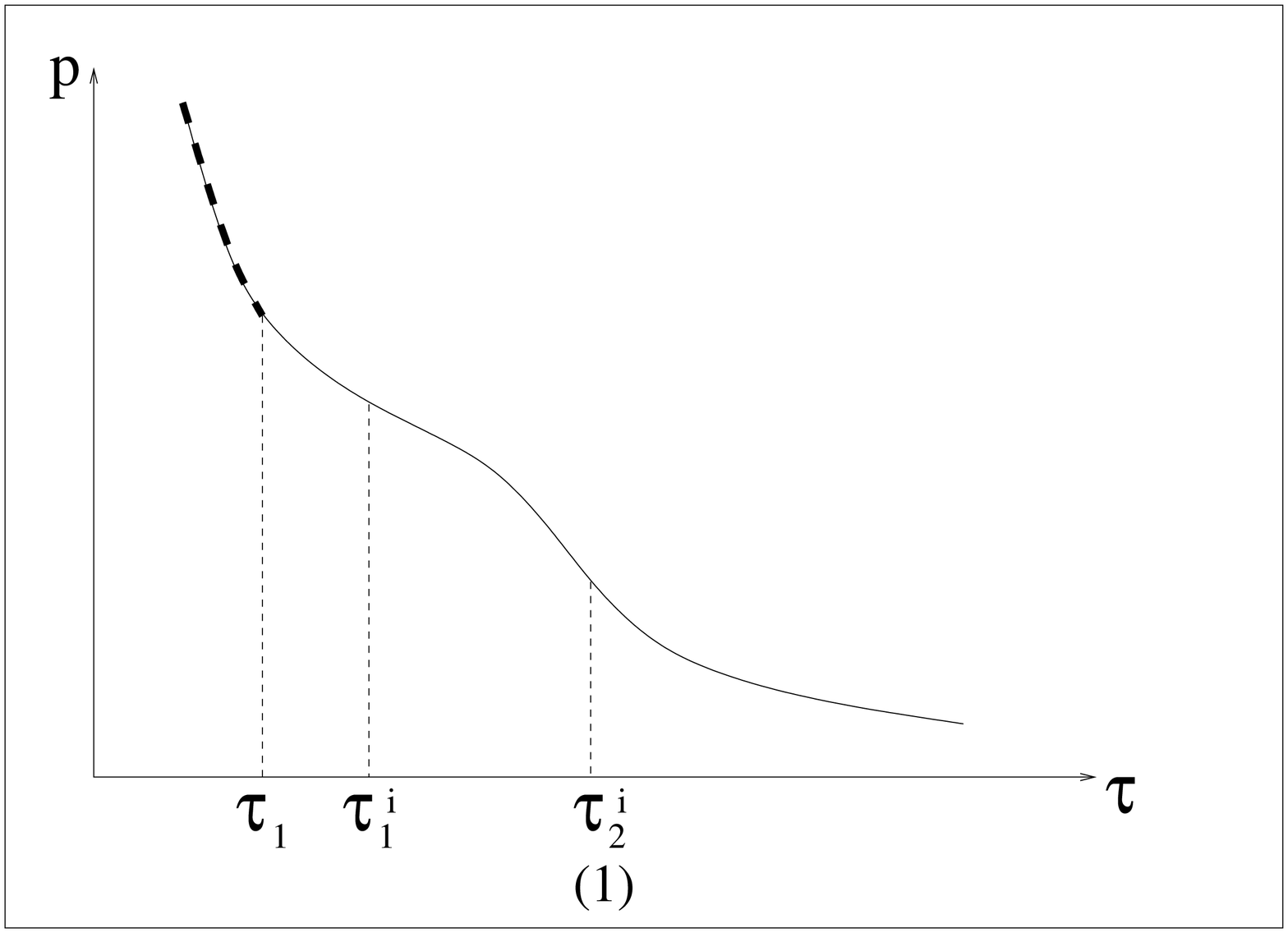}~\includegraphics[scale=0.25]{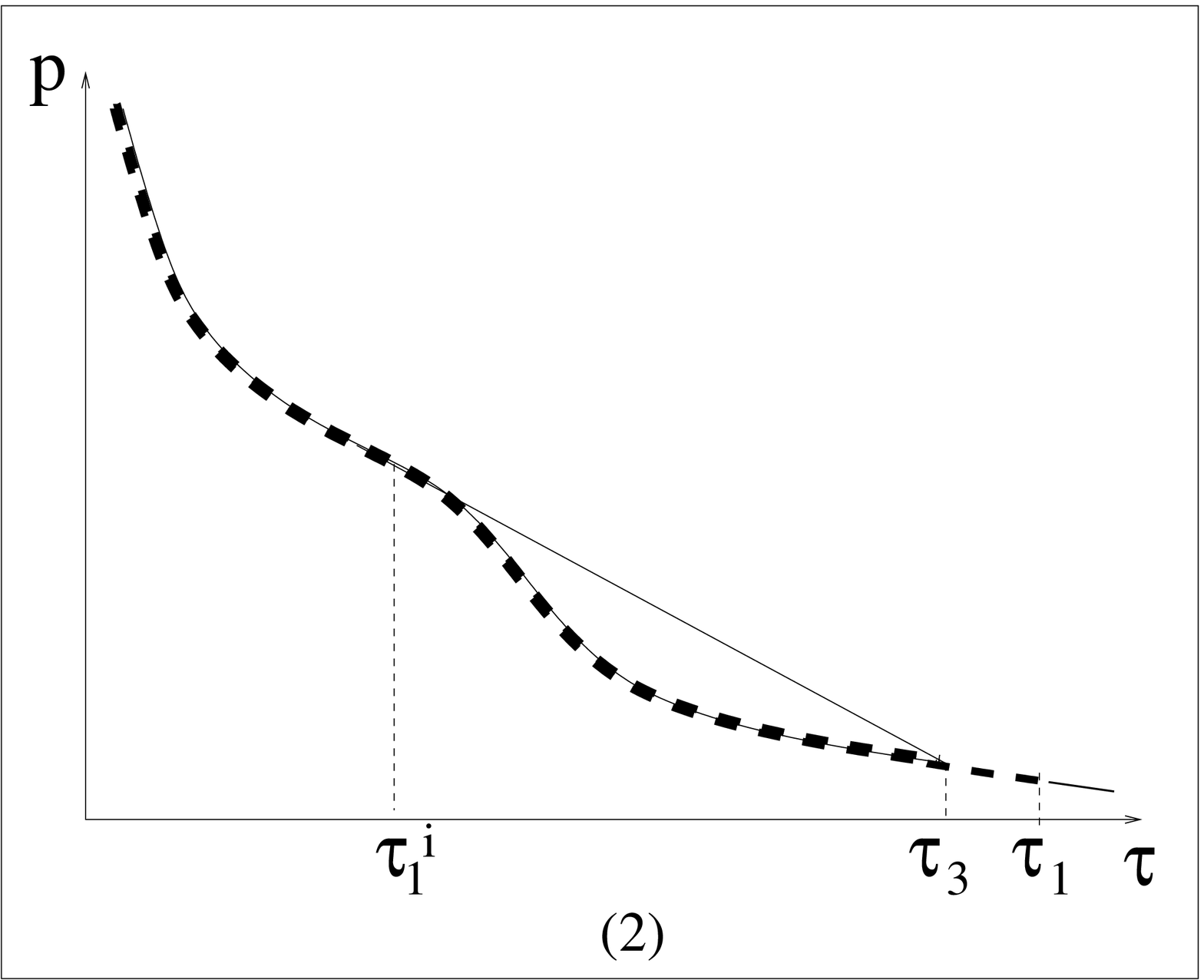}
~\includegraphics[scale=0.25]{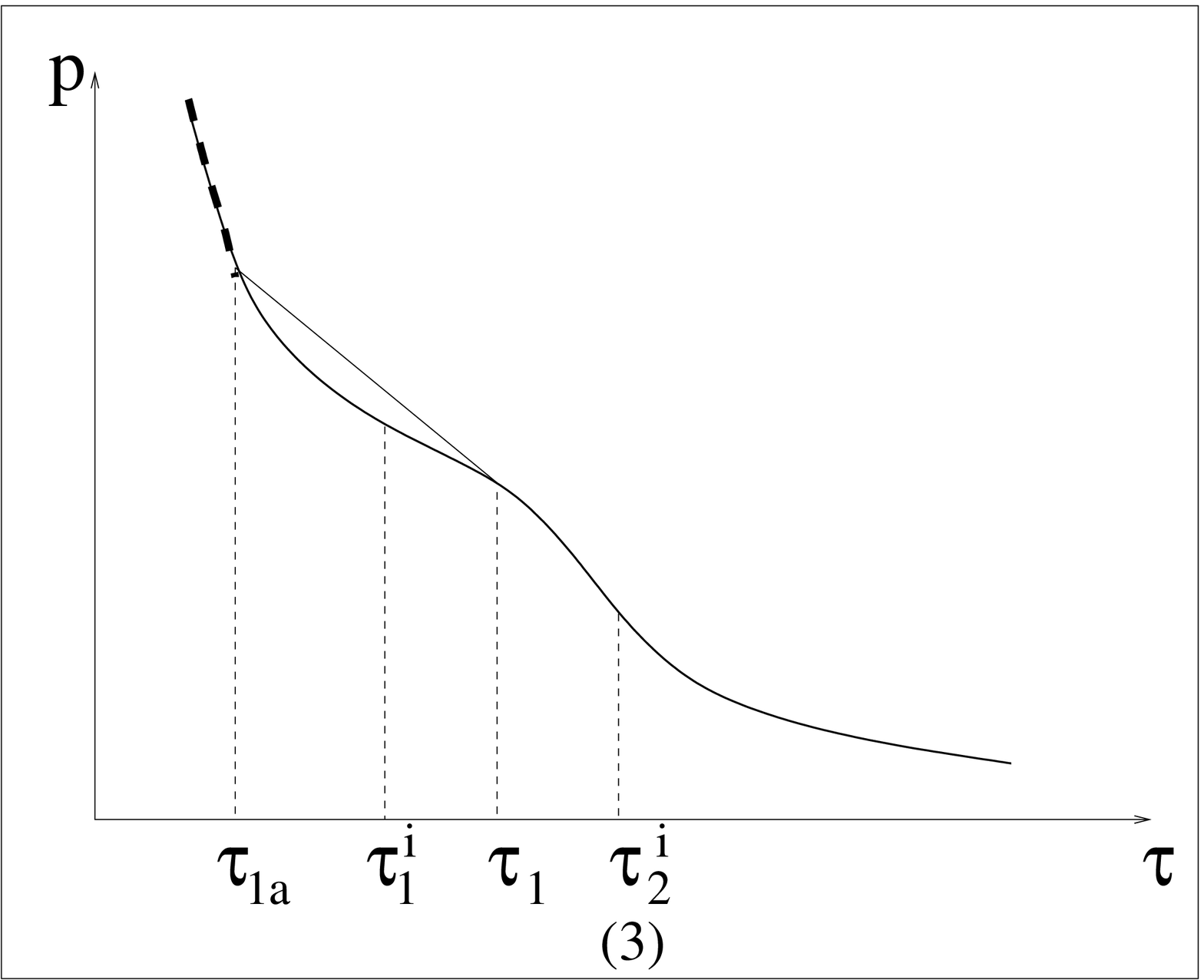}\\ \vskip 4pt
\includegraphics[scale=0.26]{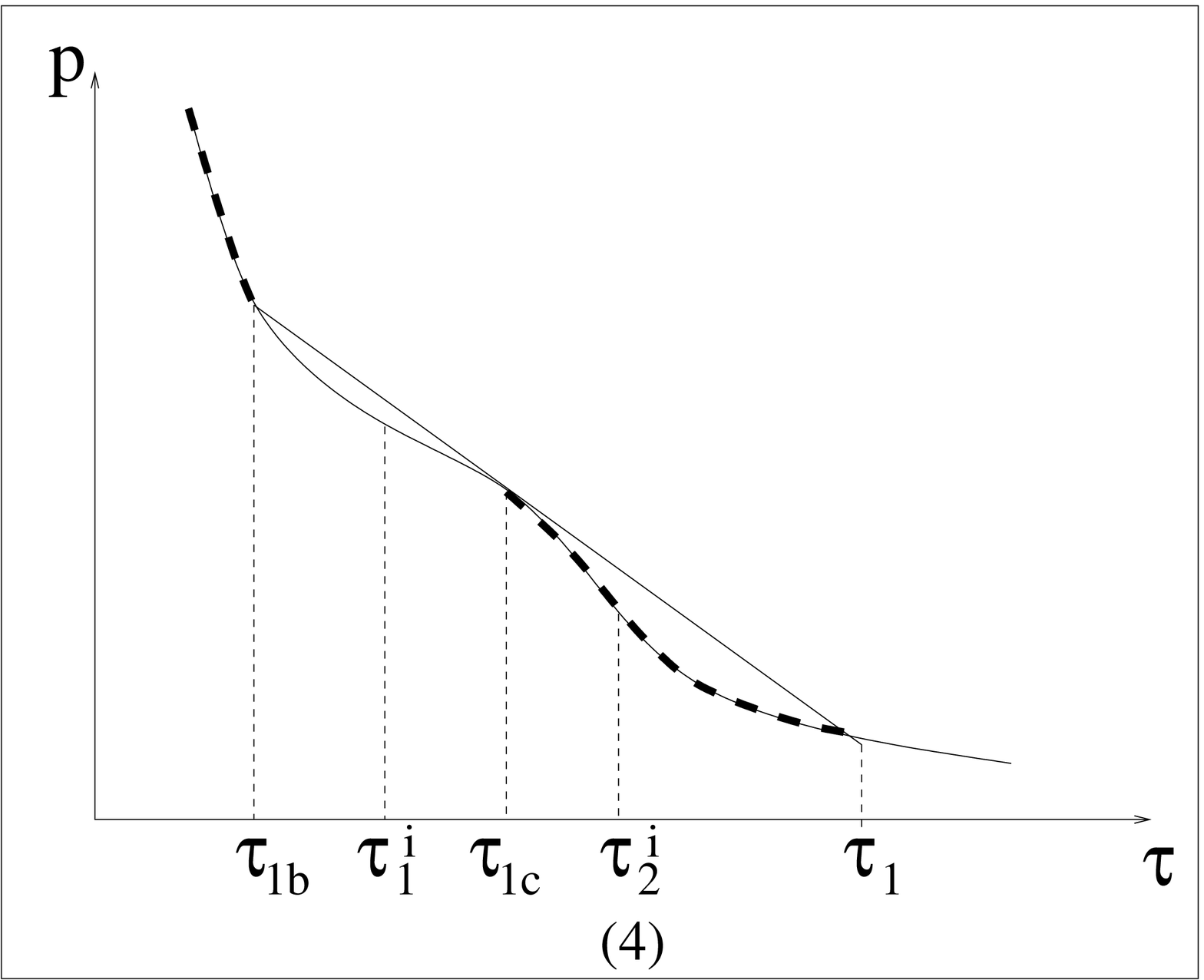}~\includegraphics[scale=0.26]{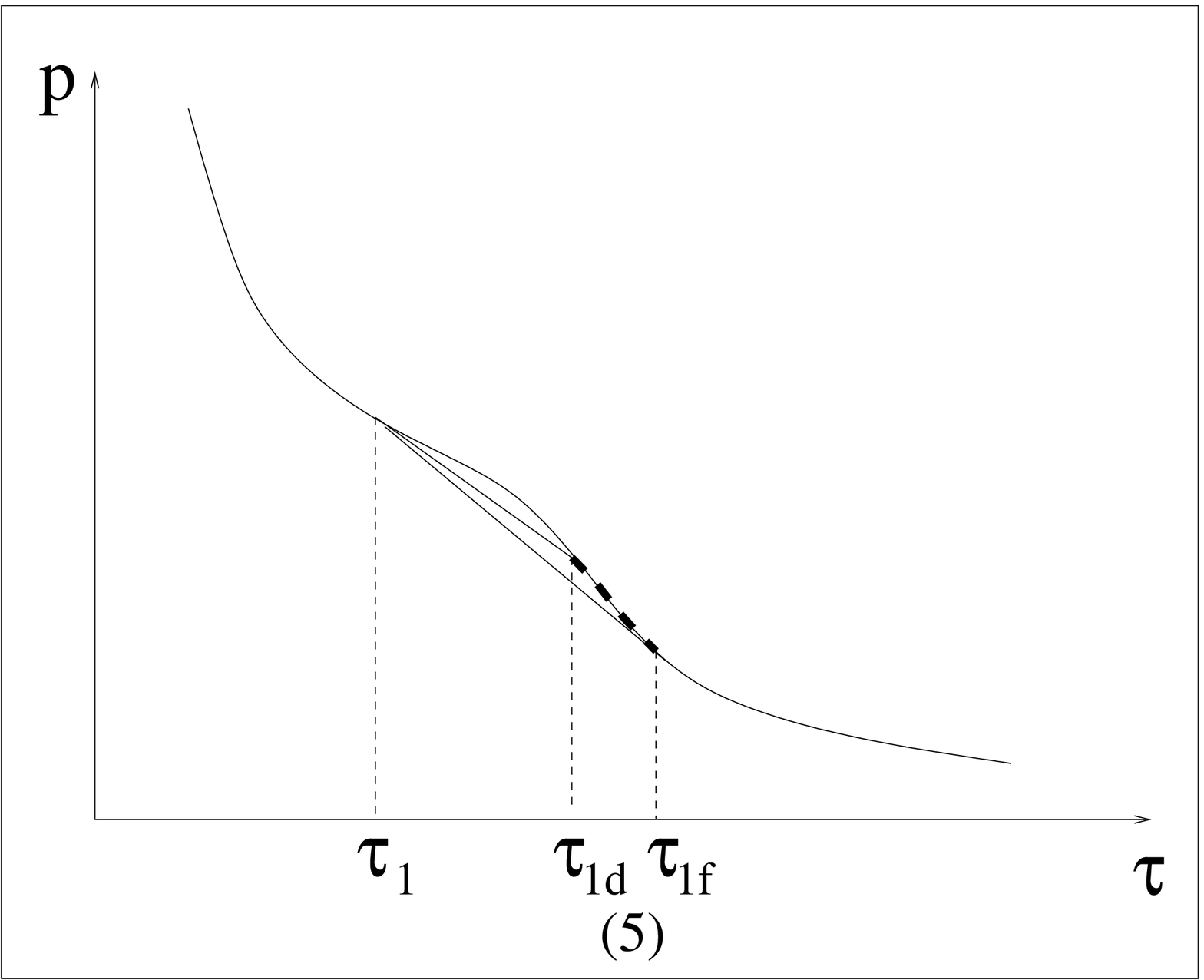}~\includegraphics[scale=0.26]{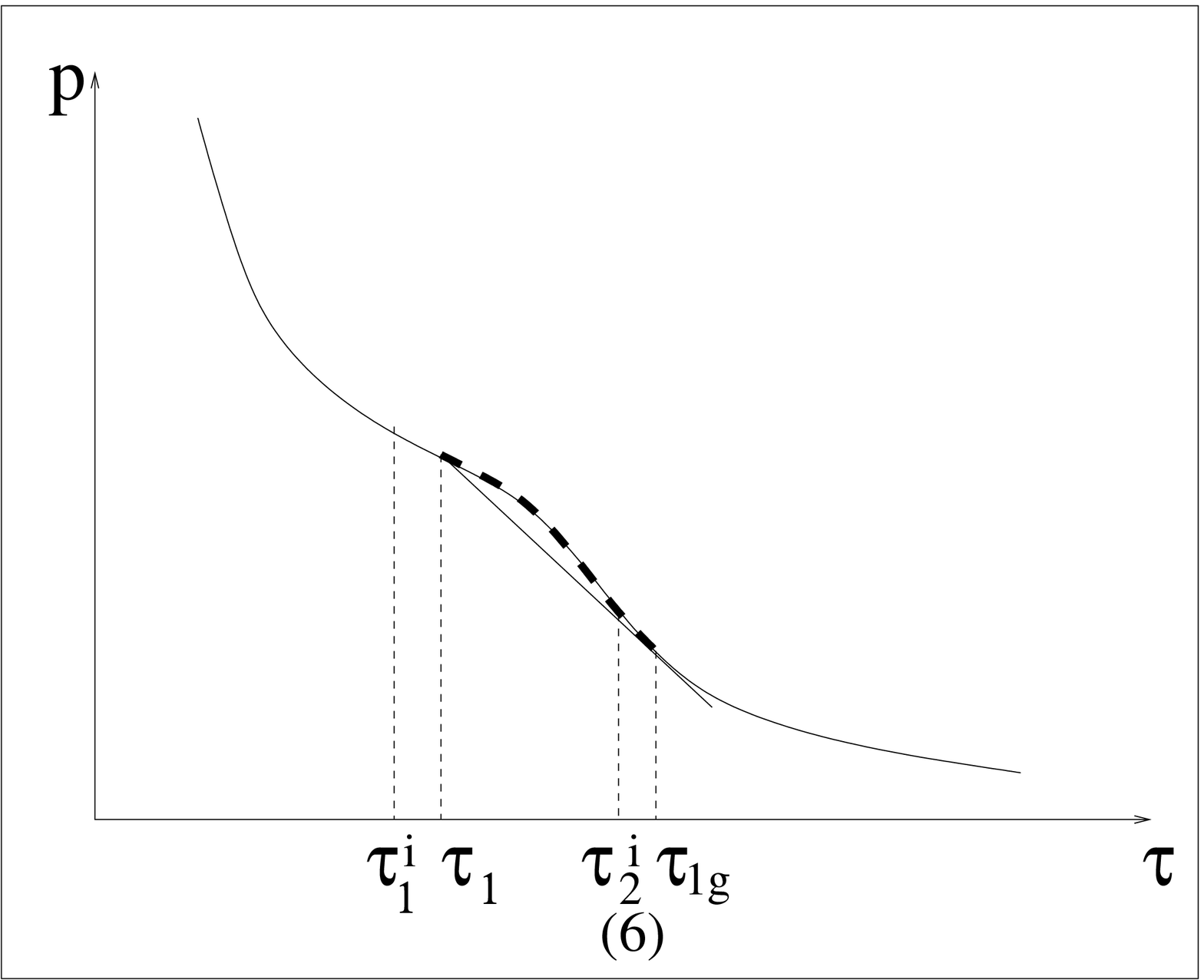}
\caption{\footnotesize Admissible shocks for equation of state II.}
\label{Figure2}
\end{center}
\end{figure}

\begin{cor}\label{cor2}
For equation of state II, the set $\mathcal{S}_{c}$ of the sates which can be connected to $(u_1, \tau_1)$ by a forward admissible compression shock on the left is given by:
\begin{itemize}
  \item If $\tau_1\in (0, \tau_1^i]\cup (\tau_3, +\infty)$, then $\mathcal{S}_{c}=\{(u, \tau)\mid u=\phi(\tau; u_1, \tau_1), \tau<\tau_1\}$, where $\tau_3$ is determined by $$\frac{p(\tau_{3})-p(\tau_{1}^{i})}{\tau_{3}-\tau_{1}^{i}}=p'(\tau_1^{i});$$ see Figure \ref{Figure2}(1--2).
  \item If $\tau_1\in (\tau_1^i, \tau_2^i]$, then $\mathcal{S}_{c}=\{(u, \tau)\mid u=\phi(\tau, u_1, \tau_1), \tau<\tau_{1a}\}$, where $\tau_{1a}$ is determined by $$\frac{p(\tau_{1a})-p(\tau_1)}{\tau_{1a}-\tau_1}=p'(\tau_1)\quad \mbox{and}\quad \tau_{1a}<\tau_1;$$ see Figure \ref{Figure2}(3).
  \item If $\tau_1\in (\tau_2^i, \tau_3)$, then $\mathcal{S}_{c}=\{(u, \tau)\mid u=\phi(\tau; u_1, \tau_1), \tau<\tau_{1b}~\mbox{and}~\tau_{1c}<\tau<\tau_1\}$, where $\tau_{1b}$ and $\tau_{1c}$ are determined by
$$\frac{p(\tau_{1b})-p(\tau_1)}{\tau_{1b}-\tau_1}=\frac{p(\tau_{1c})-p(\tau_1)}{\tau_{1c}-\tau_1}=p'(\tau_{1c})\quad \mbox{and}\quad \tau_{1b}<\tau_1^{i}<\tau_{1a}<\tau_2^{i};$$ see Figure \ref{Figure2}(4).
\end{itemize}
\end{cor}

\begin{cor}\label{cor3}
For equation of state II, the set $\mathcal{S}_{r}$ of the sates which can be connected to $(u_1, \tau_1)$ by a forward admissible rarefaction shock on the left is given by:
\begin{itemize}
  \item If $\tau_1\in (\hat{\tau}_1, \tau_1^i]$, then $\mathcal{S}_{r}=\{(u, \tau)\mid u=\phi(\tau; u_1, \tau_1), \tau_{1d}<\tau<\tau_{1f}\}$, where $\tau_{1d}$ and $\tau_{1f}$ are determined by
$$\frac{p(\tau_{1d})-p(\tau_1)}{\tau_{1d}-\tau_1}=p'(\tau_{1}), \quad
\frac{p(\tau_{1f})-p(\tau_1)}{\tau_{1f}-\tau_1}=p'(\tau_{1f}),
\quad \mbox{and}\quad \tau_1^{i}<\tau_{1d}<\tau_{1f};$$ see Figure \ref{Figure2}(5).
  \item If $\tau_1\in (\tau_1^i, \tau_2^i)$, then $\mathcal{S}_{r}=\{(u, \tau)\mid u=\phi(\tau; u_1, \tau_1), \tau_{1}<\tau<\tau_{1g}\}$, where $\tau_{1g}$ is determined by $$\frac{p(\tau_{1g})-p(\tau_1)}{\tau_{1g}-\tau_1}=p'(\tau_{1g})\quad \mbox{and}\quad \tau_{1g}>\tau_{2}^{i};$$ see Figure \ref{Figure2}(6).
\end{itemize}
\end{cor}

\begin{cor}\label{cor4}
For equation of state III, the set $\mathcal{S}_{c}$ of the sates which can be connected to $(u_1, \tau_1)$ by a forward admissible compression shock on the left is given by:
\begin{itemize}
  \item If $\tau_1\in (0, \tilde{\tau}_2]$, then $\mathcal{S}_{c}=\{(u, \tau)\mid u=\phi(\tau; u_1, \tau_1), \tau<\min\{\tilde{\tau}_1, \tau_{1}\}\}$; see Figure \ref{fignew}(1).
  \item If $\tau_1\in (\tilde{\tau}_2, +\infty)$, then $\mathcal{S}_{c}=\{(u, \tau)\mid u=\phi(\tau; u_1, \tau_1), 0<\tau<\tau_{1h}~\mbox{and}~\tilde{\tau}_{2}<\tau<\tau_{1}\}$, where $\tau_{1h}$ is determined by $$\frac{p(\tau_{1h})-p(\tau_1)}{\tau_{1h}-\tau_1}=\frac{p(\tilde{\tau}_{2})-p(\tau_1)}{\tilde{\tau}_{2}-\tau_1};$$ see Figure \ref{fignew}(2).
\end{itemize}
\end{cor}

\begin{cor}\label{cor5}
For equation of state III, if $\tau_1\in (\tilde{\tau}_1, \tilde{\tau}_2)$, then the set $\mathcal{S}_{r}$ of the sates which can be connected to $(u_1, \tau_1)$ by a forward admissible rarefaction shock on the left is given by
 $$\mathcal{S}_{r}=\{(u, \tau)\mid u=\phi(\tau; u_1, \tau_1), \tilde{\tau}_2<\tau<\tau_{1i}\}$$
where $\tau_{1i}$ is determined by $$\frac{p(\tau_{1i})-p(\tau_1)}{\tau_{1i}-\tau_1}=p'(\tau_{1i});$$ see Figure \ref{fignew}(3).
\end{cor}

\begin{figure}[htbp]
\begin{center}
\includegraphics[scale=0.25]{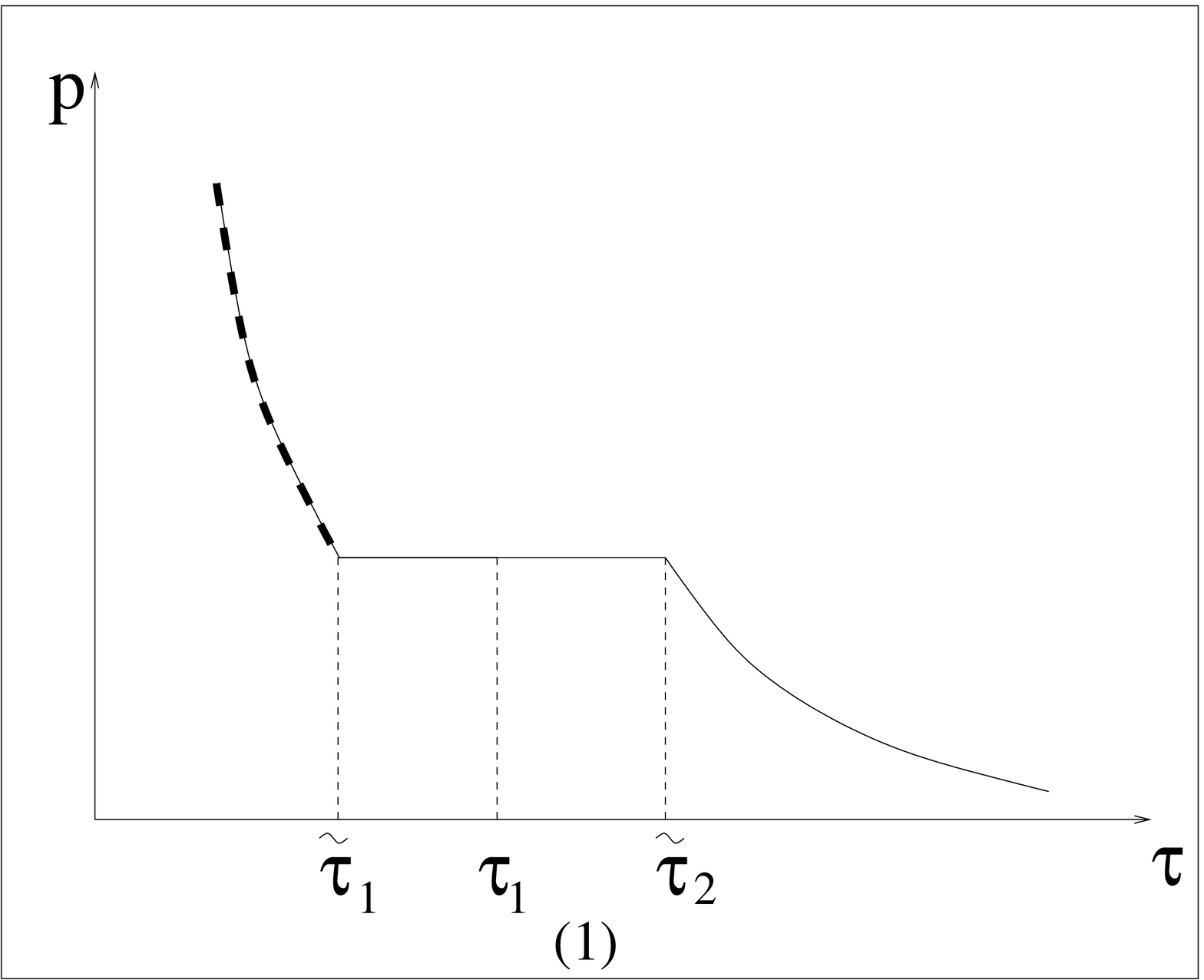}~~\includegraphics[scale=0.25]{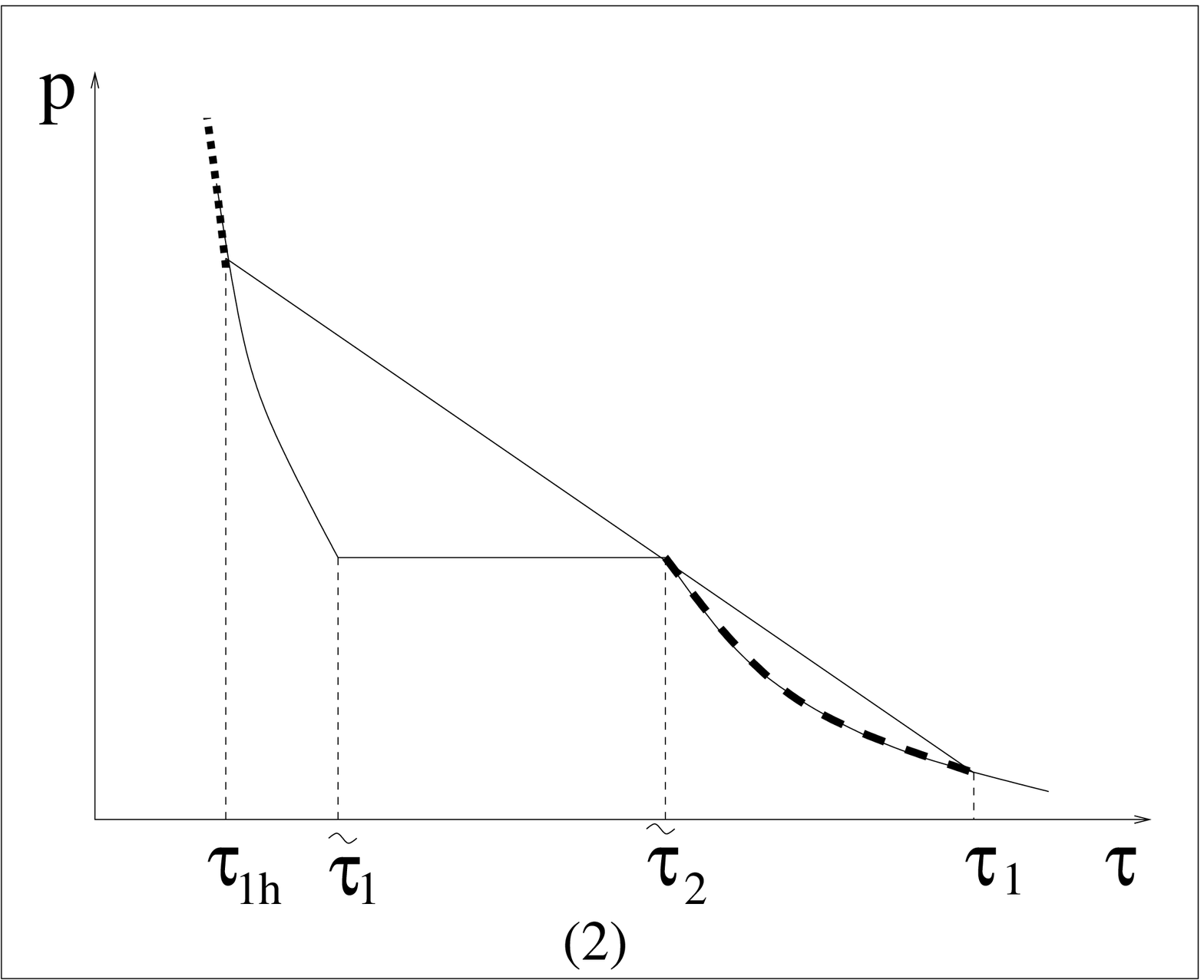}
~~\includegraphics[scale=0.25]{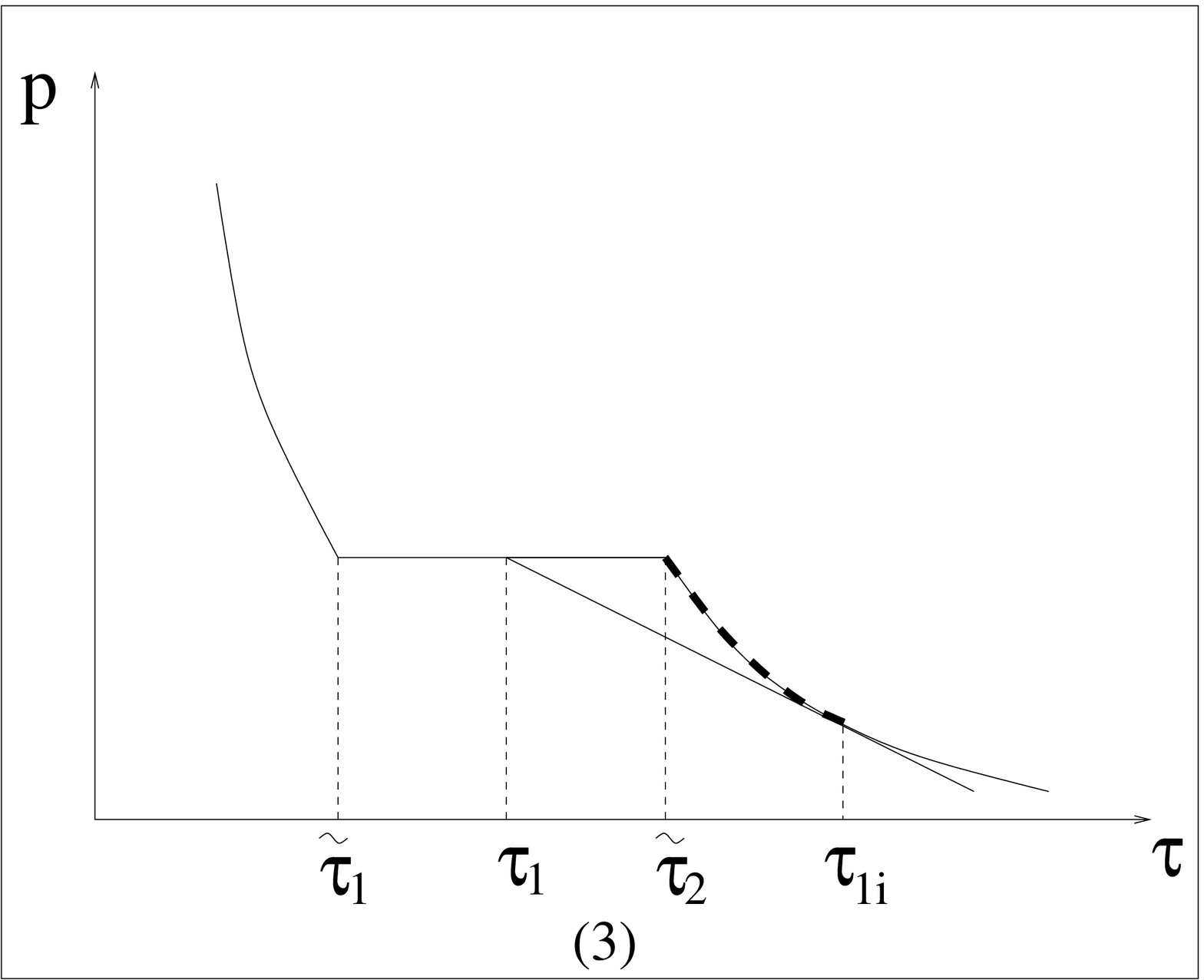}
\caption{\footnotesize Admissible shocks for equation of state III.}
\label{fignew}
\end{center}
\end{figure}

Corollaries \ref{cor1}--\ref{cor5} are obviously, so we omit their proofs. The readers can also see
 \cite{GN,LeFloch} for more details.

\section{\bf Self-similar solutions for $u_0>0$}

In this section, we will construction the self-similar solutions of the problem (\ref{AE}), (\ref{IBV})  for $u_0>0$.

\subsection{Equation of state I}
\begin{lem}\label{lem3.1}
For any fixed $S>0$, if the initial value problem (\ref{ODE2}), (\ref{ID1}) has a solution $(u_1, \rho_1)(s)$ in $(0, S)$ and $s^2p'(\rho_1)-(1-u_1s)^2<0$ as $0<s<S$, then
we have
$$
u_1(s)>0,\quad \rho_1(s)>0, \quad \frac{{\rm d} u_1(s)}{{\rm d} s}<0,\quad\mbox{and}\quad \frac{{\rm d} \rho_1(s)}{{\rm d} s}<0\quad\mbox{as}\quad 0<s<S.
$$
\end{lem}
\begin{proof}
Assume there exists a $0<s_*<S$ such that $u_1(s)>0$ as $0<s<s_*$ and $u_1(s_*)=0$. Then by (\ref{ODE2}) we have
$$
\int_{u_0}^{0}\frac{1}{u}
~{\rm d}u=\int_{0}^{s_*}\frac{2p'(\rho_1(s)) s}{s^2 p'(\rho_1(s))-(1-u_1(s)s)^2}~{\rm d}s,
$$
which leads to a contradiction. Thus, we have $u_1(s)>0$ and $\frac{{\rm d} u_1(s)}{{\rm d} s}<0$ as $0<s<S$. Similarly, we can prove $\rho_1(s)>0$ as $0<s<S$.
From $s^2p'(\rho_1)-(1-u_1s)^2<0$ and $u_1(s)<0$,  we have $u_1(s)s<1$, consequently we have $\frac{{\rm d} \rho_1(s)}{{\rm d} s}<0$ as $0<s<S$. We then complete the proof of this lemma.
\end{proof}

\begin{lem}\label{lem1}
For any fixed $S>0$, if the initial value problem (\ref{ODE2}), (\ref{ID1}) has a solution $(u_1, \rho_1)(s)$ in $(0, S)$ and $\rho_1(s)>0$ and $h(\rho_1(s), s)>0$ as $0<s<S$, then $0< u_1(s)<h(\rho_1(s), s)$ as $0<s<S$.
\end{lem}
\begin{proof}
It is obviously that $0<u_1(s)<h(\rho_1(s), s)$ as $s$ is sufficiently small. By a direct computation, we have
\begin{equation}\label{100504}
\frac{{\rm }d (u_1-h)}{{\rm d}s}
~=~\frac{2p'(\rho_1) u_1s}{s^2 p'(\rho_1)-(1-u_1s)^2}+\frac{1}{s^2}+\frac{2p''(\rho_1) \rho_1 u_1(1-u_1s)}{2\sqrt{p'(\rho_1)}\big(s^2 p'(\rho_1)-(1-u_1s)^2\big)}.
\end{equation}

 Suppose that $s_0\in(0, S)$ is the ``first" point such that $\rho_1(s_0)>0$, $h(\rho_1(s_0), s_0)>0$ and $u_1(s_0)=h(\rho_1(s_0), s_0)$. Then we have
\begin{equation}\label{100505}
\begin{aligned}
&~p'(\rho_1) u_1s+\frac{p''(\rho_1) \rho_1 u_1(1-u_1s)}{2\sqrt{p'(\rho_1)}}\\=&~
\frac{u_1}{2\sqrt{-p'(\tau_1)}}\Big(-2\tau_1^2p'(\tau_1)s\sqrt{-p'(\tau_1)}+\tau_1^2p''(\tau_1)(1-u_1s)+2\tau_1 p'(\tau_1)(1-u_1s)\Big)\\=&~
\frac{u_1\tau_1^2p''(\tau_1)s\sqrt{p'(\rho_1)}}{2\sqrt{-p'(\tau_1)}}~>~0
\end{aligned}
\end{equation}
as $s=s_0$, where $\tau_1=1/\rho_1$. Here, we use $p'(\rho_1)=-\tau_1^2p'(\tau_1)$ and $p''(\rho_1)=2\tau_1^3p'(\tau_1)+\tau_1^4p''(\tau_1)$.
From $0<u_1(s)<h(\rho_1(s), s)$ as $s<s_0$, we get $s^2 p'(\rho_1)-(1-u_1s)^2<0$ as $s<s_0$.
Hence, we have $\lim\limits_{s\rightarrow s_0^{-}} \frac{{\rm }d (u_1-h)}{{\rm d}s}=-\infty$ which leads to a contradiction.
We then complete the proof of the lemma.
\end{proof}

In what follows, we are going to show that there are the only following three cases for the local solution $(u_1, \rho_1)(s)$:
\begin{itemize}
  \item There exists a $s_*>0$ such that $0<u_1(s)<h(\rho_1(s), s)$ as $s<s_*$ and $h(\rho_1(s_*), s_*)=u_1(s_*)=1/s_{*}$; see Figure \ref{fig1}(left).
  \item There exists a $s_{*}>0$ such that  $0<u_1(s)<h(\rho_1(s), s)$ as $s<s_*$ and  $u_1(s_*)=h(\rho_1(s_*), s_*)=0$; see Figure \ref{fig2}(left).
  \item $0<u_1(s)<h(\rho_1(s), s)$ for all $s>0$; see Figure \ref{fig3}(left).
\end{itemize}

\begin{lem}
If $u_0>0$ is sufficiently large, then there exists a $s_*>0$ such that  $0<u_1(s)<h(\rho_1(s), s)$ as $s<s_*$ and $h(\rho_1(s_*), s_*)=u_1(s_*)=1/s_{*}$.
\end{lem}
\begin{proof}
If $u_1(s)<h(\rho_1(s), s)$ then we have
$$
u_1(s)s<1\quad \mbox{and}\quad s^2p'(\rho_1(s))-(1-u_1(s)s)^2<0.
$$
Consequently by (\ref{ODE2}) we have
$$\frac{{\rm d} \rho_1}{{\rm d} u_1}~=~\frac{\rho_1 (1-u_1s)}{p'(\rho_1)s}~>~\frac{\rho_1}{\sqrt{p'(\rho_1)}}.$$
Integrating this, we get
\begin{equation}\label{91201}
\int_{0}^{\rho_0}\frac{\sqrt{p'(\rho_1)}}{\rho_1}~{\rm d}\rho_1~\geq~\int_{\rho_1(s)}^{\rho_0}\frac{\sqrt{p'(\rho_1)}}{\rho_1}~{\rm d}\rho_1~>~\int_{u_1(s)}^{u_0}~{\rm d}u_1.
\end{equation}
Combining this with assumption (A1) and Lemmas \ref{lem3.1} and \ref{lem1},  we know that when $u_0$ is sufficiently large, e.g. $u_0>\int_{0}^{\rho_0}\frac{\sqrt{p'(\rho)}}{\rho}~{\rm d}\rho$,
 there exists a $s_*>0$ such that
$\rho_1(s_*)=0$ and $u_1(s_*)=h(\rho_1(s_*), s_*)=1/s_*$.

We then have this lemma.
\end{proof}

\begin{figure}[htbp]
\begin{center}
\includegraphics[scale=0.42]{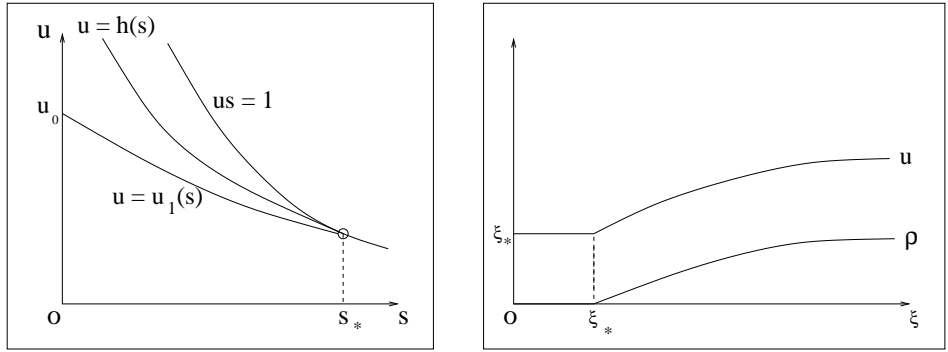}
\caption{\footnotesize Continuous solution with a vacuum, where $\xi=x/t$.}
\label{fig1}
\end{center}
\end{figure}
Therefore, the self-similar solution of the problem (\ref{AE}), (\ref{IBV}) for this case has the form
$$
(u, \rho)(x, t)=\left\{
                 \begin{array}{ll}
                   (u_1, \rho_1)(s), & \hbox{$s<s_*$,} \\[4pt]
                   \big(\xi_*, 0\big), & \hbox{$s>s_*$;}
                 \end{array}
               \right.
$$
where $s=t/x$ and $\xi_*=1/s_*$.
This is a continuous solution with a growing vacuum region; see Figure \ref{fig1}(right).

\begin{lem}
If $u_0>0$ is sufficiently small then there exists a $s_{*}>0$ such that  $0<u_1(s)<h(\rho_1(s), s)$ as $s<s_*$ and  $u_1(s_*)=h(\rho_1(s_*), s_*)=0$.
\end{lem}
\begin{proof}
Let $\varepsilon\in (0, \rho_0/2)$ be given such that
\begin{equation}\label{91202}
\max\limits_{\rho\in [\rho_0-\varepsilon, \rho_0]}\sqrt{p'(\rho)}~<~\frac{3}{2}\sqrt{p'(\rho_0)}.
\end{equation}
Let $s_0=\frac{1}{2\sqrt{p'(\rho_0)}}$.
It follows from (\ref{ODE2}) that if $u_0>0$ is sufficiently small then $\rho_1(s_0)>\rho_0-\varepsilon$.
Consequently, by (\ref{91202}) we have $h(\rho_1(s_0), s_0)>0$.

From (\ref{ODE2}), we have
$$\frac{{\rm d} \rho_1}{{\rm d} u_1}~=~\frac{\rho_1 (1-u_1s)}{p'(\rho_1)s}~<~\frac{\rho_1}{s_0p'(\rho_1)}\quad \mbox{as}\quad s>s_0.$$
Hence, we have
\begin{equation}\label{91204}
\int_{\rho_1(s)}^{\rho_0-\varepsilon}\frac{p'(\rho_1)}{\rho_1}~{\rm d}\rho_1~<~\int_{\rho_1(s)}^{\rho_1(s_0)}\frac{p'(\rho_1)}{\rho_1}~{\rm d}\rho_1~<~\frac{1}{s_0}\int_{u_1(s)}^{u_1(s_0)}{\rm d}u_1~<~\frac{u_0}{s_0}.
\end{equation}
Thus, when $u_0$ is sufficiently small there exists a $\rho_m>0$ such that
$\rho_1(s)>\rho_m$.
Therefore, there must exists a $s_*>0$ such that $h(\rho_1(s_*), s_*)=0$. By Lemma \ref{lem1} we also have $u_1(s_*)=0$.
We then complete the proof of this lemma.
\end{proof}

Therefore, the self-similar solution of the problem (\ref{AE}), (\ref{IBV}) for this case has the form
$$
(u, \rho)(x, t)=\left\{
                 \begin{array}{ll}
                   (u_1, \rho_1)(s), & \hbox{$s<s_*$,} \\[4pt]
                   \big(0, \rho_1(s_*)\big), & \hbox{$s>s_*$;}
                 \end{array}
               \right.
$$
where $s=t/x$.
This is a continuous solution with a quiet constant state; see Figure \ref{fig2}(right).

\begin{figure}[htbp]
\begin{center}
\includegraphics[scale=0.4]{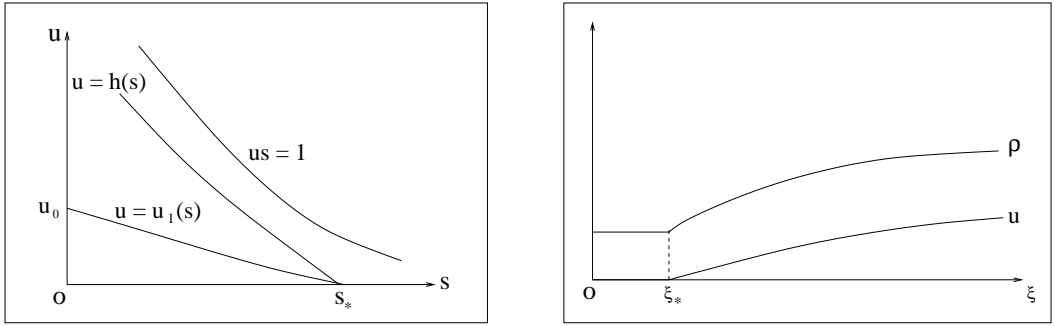}
\caption{\footnotesize Continuous solution with a quiet constant state.}
\label{fig2}
\end{center}
\end{figure}

\begin{lem}\label{lemcase31}
If the first case happens as $(u, \rho)(0)=(u^0, \rho^0)$, then there exists a sufficiently small $\varepsilon>0$ such that the first case will happen as $(u, \rho)(0)\in({u}^0-\varepsilon, {u}^0+\varepsilon)\times\{\rho^0\}$.
\end{lem}
\begin{proof}
Denote by $(\bar{u}, \bar{\rho})(s)$ the solution of system (\ref{ODE2}) with
the initial data $(u, \rho)(0)=(u^0, \rho^0)$. Then there exists a $\bar{s}_*>0$ such that $\bar{\rho}(\bar{s}_*)=0$ and $\bar{u}(\bar{s}_{*})=h(\bar{\rho}(\bar{s}_{*}), \bar{s}_{*})=1/\bar{s}_{*}$.

From assumption (A1), we can find a sufficiently small $\delta\in (0, 1/\bar{s}_{*})$ such that
\begin{equation}\label{91203}
\int_{0}^{\delta}\frac{\sqrt{p'(\rho)}}{\rho}~{\rm d}\rho~<~\frac{1}{2\bar{s}_{*}}.
\end{equation}

Since $\bar{\rho}(s)$ is continuous on $[0, \bar{s}_{*}]$, there exists a sufficiently small $\eta>0$ such that
\begin{equation}
\bar{\rho}(\bar{s}_{*}-\eta)<\frac{\delta}{2}.
\end{equation}
When $\varepsilon>0$ is sufficiently small the solution $(u, \rho)(s)$ of system (\ref{ODE2}) with the initial data
 $(u, \rho)(0)~\in~ ({u}^0-\varepsilon, {u}^0+\varepsilon)\times\{\rho^0\}$ satisfies
\begin{equation}\label{8506}
|\rho(\bar{s}_{*}-\eta)-\bar{\rho}(\bar{s}_{*}-\eta)|<\frac{\delta}{4}\quad\mbox{and}\quad
|u(\bar{s}_{*}-\eta)-\bar{u}(\bar{s}_{*}-\eta)|<\frac{\delta}{4}.
\end{equation}

It is similar to (\ref{91201}) that
$$
\int_{\rho(s)}^{\rho(\bar{s}_{*}-\eta)}\frac{\sqrt{p'(\rho)}}{\rho}~{\rm d}\rho~>~\int_{u(s)}^{u(\bar{s}_{*}-\eta)}~{\rm d}u=u(\bar{s}_{*}-\eta)-u(s)
$$
as $s>\bar{s}_{*}-\eta$.
Combining this with (\ref{8506}), we get
$$
\begin{aligned}
\int_{\rho(s)}^{\delta}\frac{\sqrt{p'(\rho)}}{\rho}~{\rm d}\rho~&>~\bar{u}(\bar{s}_{*}-\eta)-\frac{\delta}{4}-u(s)>~\bar{u}(\bar{s}_{*})-\frac{\delta}{4}-u(s)\\&~=~\frac{1}{\bar{s}_{*}}-\frac{\delta}{4}-u(s)~>~
\frac{3}{4\bar{s}_{*}}-u(s)
\end{aligned}
$$
as $s>\bar{s}_{*}-\eta$.
Thus, by (\ref{91203}) and Lemmas \ref{lem3.1} and \ref{lem1} we know that
 there exists a $s_{*}$ such that $u(s)<h(\rho(s), s)$ as $0<s<s_{*}$ and $u(s_*)=h(\rho(s_*), s_*)=1/s_*$. We then have this lemma.
\end{proof}

\begin{lem}\label{lemcase32}
 If the second case happens as $(u, \rho)(0)=(u^0, \rho^0)$, then there exists a sufficiently small $\varepsilon>0$ such that the second case will happen as $(u, \rho)(0)\in({u}^0-\varepsilon, {u}^0+\varepsilon)\times\{\rho^0\}$.
\end{lem}
\begin{proof}
 Denote by $(\bar{u}, \bar{\rho})(s)$ the solution of system (\ref{ODE2}) with
the initial data $(u, \rho)(0)=(u^0, \rho^0)$. Then there exists a $\bar{s}_*>0$ such that $\bar{u}(\bar{s}_*)=0$ and $\bar{\rho}(\bar{s}_{*})=\rho_{*}>0$.

Let $$\mathcal{N}=\int_{0}^{\frac{\rho_*}{2}}\frac{p'(\rho)}{\rho}~{\rm d} \rho.$$
There exists a sufficiently small $\eta<\frac{\bar{s}_*}{2}$ such that
\begin{equation}\label{100404}
0<\bar{u}(\bar{s}_{*}-\eta)<\frac{\mathcal{N}\bar{s}_*}{4}.
\end{equation}

It is easy to see that if $\varepsilon>0$ is sufficiently small, then the solution $(u, \rho)(s)$ of  system (\ref{ODE2}) with the initial data
 $(u, \rho)(0)~\in~ ({u}^0-\varepsilon, {u}^0+\varepsilon)\times\{\rho^0\}$ satisfies
\begin{equation}\label{100403}
|\rho(\bar{s}_{*}-\eta)-\bar{\rho}(\bar{s}_{*}-\eta)|<\frac{\rho_*}{4}\quad\mbox{and}\quad
|u(\bar{s}_{*}-\eta)-\bar{u}(\bar{s}_{*}-\eta)|<\frac{\mathcal{N}\bar{s}_*}{4}.
\end{equation}

It is similar to (\ref{91204}) that
\begin{equation}
\int_{\rho(s)}^{\rho(\bar{s}_{*}-\eta)}\frac{p'(\rho)}{\rho}~{\rm d}\rho~<~\frac{1}{\bar{s}_{*}-\eta}\int_{u(s)}^{u(\bar{s}_{*}-\eta)}{\rm d}u\quad \mbox{as}\quad s>\bar{s}_{*}-\eta.
\end{equation}
Combining this with (\ref{100403}) we get
$$
\int_{\rho(s)}^{\frac{3\rho_*}{4}}\frac{p'(\rho)}{\rho}~{\rm d}\rho~<~\frac{1}{\bar{s}_{*}-\eta}\int_{u(s)}^{\frac{\mathcal{N}\bar{s}_*}{2}}{\rm d}u~<~\mathcal{N},
$$
since $\rho'(s)<0$.
Thus, by (\ref{100404}) we know that there exists a $\rho_m>0$ such that $\rho(s)>\rho_m$ as $s>\bar{s}_{*}-\eta$.
Consequently, there exists a $s_{*}>\bar{s}_{*}-\eta$ such that $h(\rho(s_*), s_{*})=0$.
We then have this lemma.
\end{proof}

Using Lemmas \ref{lemcase31} and \ref{lemcase32} and the argument of continuity, we know that for any given $\rho_0>0$ there exists a $u_0>0$ such that the solution $(u_1, \rho_1)(s)$  of the initial value problem (\ref{ODE2}), (\ref{ID1}) satisfies $0<u_1(s)<h(\rho_1(s), s)<1/s$ as $s>0$; see Figure \ref{fig3}(left). That is to say, the initial value problem (\ref{ODE2}), (\ref{ID1}) has a global classical solution. Moreover, this solution satisfies $\lim\limits_{s\rightarrow +\infty}u_1(s)=\lim\limits_{s\rightarrow +\infty}\rho_1(s)=0$. In this case, the initial-boundary value problem (\ref{AE}), (\ref{IBV}) has a self-similar smooth solution $(u, \rho)(x, t)=(u_1, \rho_1)(t/x)$; see Figure \ref{fig3}(right).

\begin{figure}[htbp]
\begin{center}
\includegraphics[scale=0.38]{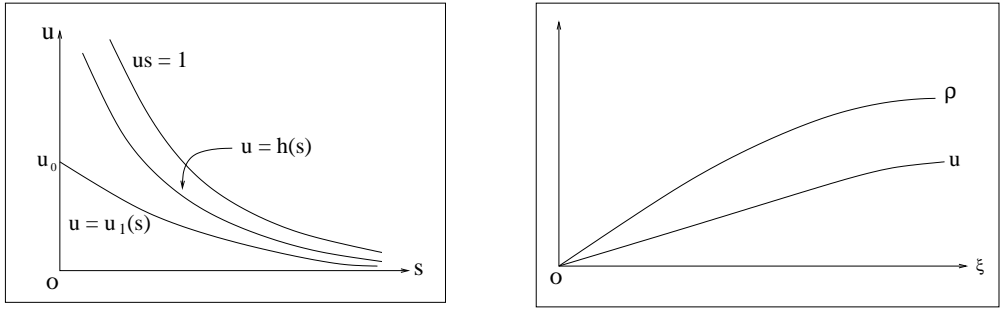}
\caption{\footnotesize A global smooth self-similar solution.}
\label{fig3}
\end{center}
\end{figure}

\subsection{Equation of state II}

\subsubsection{$\tau_0\geq\tau_2^i$}
The discussion is similar to that of section 3.1, since $\tau_1'(s)>0$ as $s>0$  and $p''(\tau)>0$ as $\tau<\tau_0$.

\subsubsection{$\tau_1^i\leq \tau_0<\tau_2^i$}
There are the following four cases:
\begin{itemize}
   \item There exists a $s_*>0$ such that  $0<u_1(s)<h(\rho_1(s), s)$ as $0<s<s_*$ and $h(\rho_1(s_*), s_*)=u_1(s_*)=\frac{1}{s_{*}}$.
  \item There exists a $s_{*}>0$ such that  $0<u_1(s)<h(\rho_1(s), s)$ as $0<s<s_*$ and $u_1(s_*)=h(\rho_1(s_*), s_*)=0$.
  \item $0<u_1(s)<h(\rho_1(s), s)$ for all $s>0$.
\item There exists a $s_*>0$ such that $0<u_1(s)<h(\rho_1(s), s)$ as $0<s<s_*$ and $0<u_1(s_*)=h(\rho_1(s_*), s_*)<\frac{1}{s_{*}}$. (By Lemma \ref{lem1}, there holds $\tau_1(s_*)\in (\tau_0, \tau_2^i)$ in this case.)
\end{itemize}
It is easy to see that the first three cases can be happened when $\tau_0$ is sufficiently close to $\tau_2^i$.
And the discussions for these three cases are similar to that of section 3.1.

In what follows, we are going to discuss the forth case. We first show that the forth case can be happened at least in some cases. To confirm this, we consider the initial value problem
\begin{equation}\label{ODE3}
\left\{
  \begin{aligned}
   &\frac{{\rm d} s}{{\rm d} u}=\frac{s^2 p'(\rho)-(1-us)^2}{2p'(\rho) us},  \\
     &\frac{{\rm d} \rho}{{\rm d} u}=\frac{\rho (1-us)}{2p'(\rho)s},
  \end{aligned}
\right.
\end{equation}
\begin{equation}\label{ID21}
(s, \rho)\mid_{u=u_*}~=~(s_*, \rho_*),
\end{equation}
where $u_*>0$, $\rho_*>0$, and $s_*>0$ satisfy $s_*^2p'(\rho_*)-(1-u_* s_*)^2=0$, $u_*s_*<1$ and $\tau_2^i<1/\rho_*<\tau_2^i$.

\begin{lem}\label{12101}
When $\delta>0$ is sufficiently small, the initial value problem (\ref{ODE3}), (\ref{ID21}) has a solution $(\hat{s}, \hat{\rho})(u)$ on $(u_*, u_*+\delta)$. Moreover, this solution satisfies $\frac{{\rm d} \hat{s}}{{\rm d} u}<0$,  $\frac{{\rm d} \hat{\rho}}{{\rm d} u}<0$, and $\frac{{\rm d}[\hat{s}^2 p'(\hat{\rho})-(1-u\hat{s})^2]}{{\rm d} u}<0$ as $u\in(u_*, u_*+\delta)$.
\end{lem}
\begin{proof}
It is easy to see that this initial value problem is a classically well-posed problem which has a unique local solution $(\hat{s}, \hat{\rho})(u)$.

By computation, we have
$$
\begin{aligned}
 &\frac{{\rm d}}{{\rm d} u}\big(\hat{s}^2 p'(\hat{\rho})-(1-u\hat{s})^2\big)\\=~&\big(2\hat{s} p'(\hat{\rho})+2u(1-u\hat{s})\big)\cdot\Big(\frac{\hat{s}^2 p'(\hat{\rho})-(1-u\hat{s})^2}{2p'(\hat{\rho}) u\hat{s}}\Big)+\hat{s}^2p''(\hat{\rho})\frac{{\rm d} \hat{\rho}}{{\rm d}u}+2\hat{s}(1-u\hat{s})\\=~&
 \big(2\hat{s} p'(\hat{\rho})+2u(1-u\hat{s})\big)\cdot\Big(\frac{\hat{s}^2 p'(\hat{\rho})-(1-u\hat{s})^2}{2p'(\hat{\rho}) u\hat{s}}\Big)+\frac{2\hat{s}\hat{\tau}^3p''(\hat{\tau})}{p'(\hat{\rho})}(1-u\hat{s})~<~0\quad \mbox{as}\quad u=u_*.
\end{aligned}
$$
 Hence, we have $\hat{s}^2 p'(\hat{\rho})-(1-u\hat{s})^2<0$ and $\frac{{\rm d}[\hat{s}^2 p'(\hat{\rho})-(1-u\hat{s})^2]}{{\rm d} u}<0$ as $u\in(u_*, u_*+\delta)$.

Moreover, by $s_*^2p'(\rho_*)-(1-u_* s_*)^2=0$ and $u_*s_*<1$  we have $1-u\hat{s}>0$ as $u\in(u_*, u_*+\delta)$.
Consequently, we have $\frac{{\rm d} \hat{s}}{{\rm d} u}<0$ and  $\frac{{\rm d} \hat{\rho}}{{\rm d} u}<0$ as $u\in(u_*, u_*+\delta)$.
\end{proof}

 When $p''(\tau)<0$ and $p'(\tau)<0$ we have $p''(\rho)=2\tau^3p'(\tau)+\tau^4p''(\tau)<0$. Hence, in view of Lemma \ref{12101}, there may exist a $u^*>u_*$ such that $\hat{s}(u^*)=0$ and $\tau_1^i<\frac{1}{\hat{\rho}(u^*)}<\tau_2^i$,  at least for some equations of state. Therefore, if we take $u_0=u^*$ and $\rho_0=\hat{\rho}(u^*)$ then the forth case will happen.

\vskip 4pt

We next construct the solution for the forth case.
From $s_*^2p'(\rho_1(s_*))-(1-u_1(s_*) s_*)^2=0$,
 we have $\lim\limits_{s\rightarrow s_*^{-}}\frac{{\rm d} u_1}{{\rm d}s}=-\infty$ and $\lim\limits_{s\rightarrow s_*^{-}}\frac{{\rm d} \rho_1}{{\rm d}s}=-\infty$; see Figure \ref{fig4}(left). This implies that the problem (\ref{AE}), (\ref{IBV}) does not have a global continuous solution. So, we need to look for a discontinuous solution.

Since $s^2p'(\rho_1)-(1-u_1 s)^2<0$ and $0<u_1s<1$ as $0<s<s_*$, we have
\begin{equation}\label{102401}
\frac{1}{s}~>~u_1(s)+\sqrt{p'(\rho_1(s))}~=~u_1(s)+\tau_1(s)\sqrt{-p'(\tau_1(s))}\quad \mbox{as}\quad0<s<s_*.
\end{equation}

We first consider the possibility of finding a compression shock wave solution.
By (\ref{102401}) and Corollary \ref{cor2} we know that for any $0<s<s_*$, there exists an admissible forward compression shock with the speed $1/s$ and the front side state $(u_1, \rho_1)(s)$. Moreover, the backside state $(u_2, \tau_2)(s)$ can be uniquely determined by
\begin{equation}\label{backside}
\left\{
  \begin{aligned}
  &\frac{1}{s}=u_1(s)+\tau_1(s)\sqrt{-\frac{p(\tau_2(s))-p(\tau_1(s))}{\tau_2(s)-\tau_1(s)}},\\
   &u_2(s)=u_1(s)+(\tau_1(s)-\tau_2(s))\sqrt{-\frac{p(\tau_2(s))-p(\tau_1(s))}{\tau_2(s)-\tau_1(s)}}.
  \end{aligned}
\right.
\end{equation}
Moreover, we have $u_2(s)>u_1(s)>0$ since $\tau_2(s)<\tau_1(s)$. By the entropy condition, we have
$$
u_2(s)-\sqrt{p'(\rho_2(s))}~<~\frac{1}{s}~<~u_2(s)+\sqrt{p'(\rho_2(s))},
$$
and consequently
\begin{equation}\label{122501}
\frac{1}{s}-\sqrt{p'(\rho_2(s))}<u_2(s)<\frac{1}{s}+\sqrt{p'(\rho_2(s))}.
\end{equation}
We now assume there exists an admissible forward compression shock with the speed $1/s_1$ and the front side state $(u_1, \rho_1)(s_1)$, where $s_1\in (0, s_*)$.
 Then we consider system (\ref{ODE2}) with the data
\begin{equation}\label{100501}
(u, \rho)\mid_{s=s_1}~ =~ (u_2, \rho_2)(s_1).
\end{equation}
We have the following lemma:
\begin{lem}\label{100502}
There exists a $s^{*}>s_1$ such that the solution $(u_3, \rho_3)(s)$ of the initial value problem (\ref{ODE2}), (\ref{100501}) satisfies
\begin{equation}\label{100506}
\frac{1}{s}-\sqrt{p'(\rho_3(s))}<u_3(s)<\frac{1}{s}+\sqrt{p'(\rho_3(s))}\quad \mbox{as}\quad s_1<s<s^{*}
\end{equation}
and $u(s^{*})=\frac{1}{s^{*}}+\sqrt{p'(\rho(s^{*}))}>\frac{1}{s_*}$.
\end{lem}
\begin{proof}
It is easy to see that if $\frac{1}{s}-\sqrt{p'(\rho_3(s))}<u_3(s)<\frac{1}{s}+\sqrt{p'(\rho_3(s))}$ then $s^2p'(\rho_3)-(1-u_3s)^2>0$.
There are two situations: $u_2(s_1)s_1\geq1$ and $u_2(s_1)s_1<1$.

If $u_2(s_1)s_1>1$,  then we have ${\rm d} u_3/{\rm d} s>0$
and ${\rm d} \rho_3/{\rm d} s<0$ as $s>s_1$.

If $u_2(s_1)s_1<1$, then we have $\rho_3'(s_1)>0$.
Using (\ref{100504}), (\ref{100505}), and the fact that $\rho_3'(s)>0$ as $u_3s<1$,
we get $u_3(s)>1/s-\sqrt{p'(\rho_3(s))}$.
Thus there exists a $s_2> s_1$ such that $u_3(s_2)s_2=1$. Moreover, we have ${\rm d} u_3/{\rm d} s>0$
and ${\rm d} \rho_3/{\rm d} s<0$ as $s>s_2$.

If the curves $u=u_3(s)$ and $u=1/s+\sqrt{p'(\rho_3(s))}$ do not intersect with each other,
then there must have
$$\lim\limits_{s\rightarrow +\infty}\rho_3(s)=\rho_{\infty}>0\quad\mbox{and}\quad \lim\limits_{s\rightarrow +\infty}u_3(s)=u_{\infty}>0.$$
Thus, we have
$$
\frac{{\rm d} u_3}{{\rm d} s}=\frac{2p'(\rho_3) u_3s}{s^2 p'(\rho_3)-(1-u_3s)^2}>\frac{2u_3(s)}{s}>\frac{2u_2(s_1)}{s}
$$
which leads to a contradiction.
We then complete the proof of this lemma.
\end{proof}
Lemma \ref{100502} implies that the initial value problem (\ref{ODE2}), (\ref{100501}) does not have a solution on $(s_1, +\infty)$. If follows from (\ref{100506}) that $(u_3, \rho_3)(s)$ ($s_1<s<s^{*}$) can not be the front side state of any admissible forward shock with the speed $1/s$. Therefore, the problem (\ref{AE}), (\ref{IBV}) does not permit a compression shock wave solution.
In what follows, we are going to look for a rarefaction shock wave solution.


For $\tau_1^i<\tau_1(s)<\tau_2^i$ we let $f(\tau_1(s))$ be defined such that
$$
\frac{p(\tau_1)-p(f(\tau_1))}{\tau_1-f(\tau_1)}=p'(f(\tau_1)) \quad \mbox{and}\quad f(\tau_1)>\tau_2^i.
$$
It can be seen that
\begin{equation}\label{102501}
-p'(\tau_1(s))<-p'(f(\tau_1(s)))\quad \mbox{as}~0<s<s_*.
\end{equation}

\begin{lem}\label{lem1701}
There exists a $s_{**}\in (0, s_*)$ such that for any
$s\in [s_{**}, s_{*}]$,
there exists an admissible forward rarefaction shock with the speed $1/s$ and the front side state $(u_1, \rho_1)(s)$.
\end{lem}
\begin{proof}
According to Corollary \ref{cor3},
in order that $(u_1, \rho_1)(s)$ can be the front side state of an admissible forward rarefaction shock with the speed $1/s$, there must holds
$$
u_1(s)+\tau_1(s)\sqrt{-p'(\tau_1(s))}~\leq~ \frac{1}{s}~\leq~ u_1(s)+\tau_1(s)\sqrt{-p'(f(\tau_1(s)))}.
$$

Since $u_1(s)<h(\rho_1(s), s)$ as $0<s<s_*$, we have
$$
\frac{1}{s}~>~u_1(s)+\tau_1(s)\sqrt{-p'(\tau_1(s))}\quad \mbox{as}\quad 0<s<s_*.
$$

From $u_1(s_*)=h(\rho_1(s_*), s_*)>0$, $\tau_0<\tau_1(s_*)<\tau_2^i$, and (\ref{102501}), we have
\begin{equation}\label{81703}
\frac{1}{s_*}~=~u_1(s_*)+\tau_1(s_*)\sqrt{-p'(\tau_1(s_*))}~<~u_1(s_*)+\tau_1(s_*)\sqrt{-p'(f(\tau_1(s_*)))}.
\end{equation}
Thus, there exists a $s_{**}\in (0, s_*)$ such that $1/s<u_1(s)+\tau_1(s)\sqrt{-p'(f(\tau_1(s)))}$ as $s_{**}<s<s_*$ and
\begin{equation}\label{102502}
\frac{1}{s_{**}}=u_1(s_{**})+\tau_1(s_{**})\sqrt{-p'(f(\tau_1(s_{**})))}.
\end{equation}
We then complete the proof of this lemma.
\end{proof}

\begin{figure}[htbp]
\begin{center}
\includegraphics[scale=0.43]{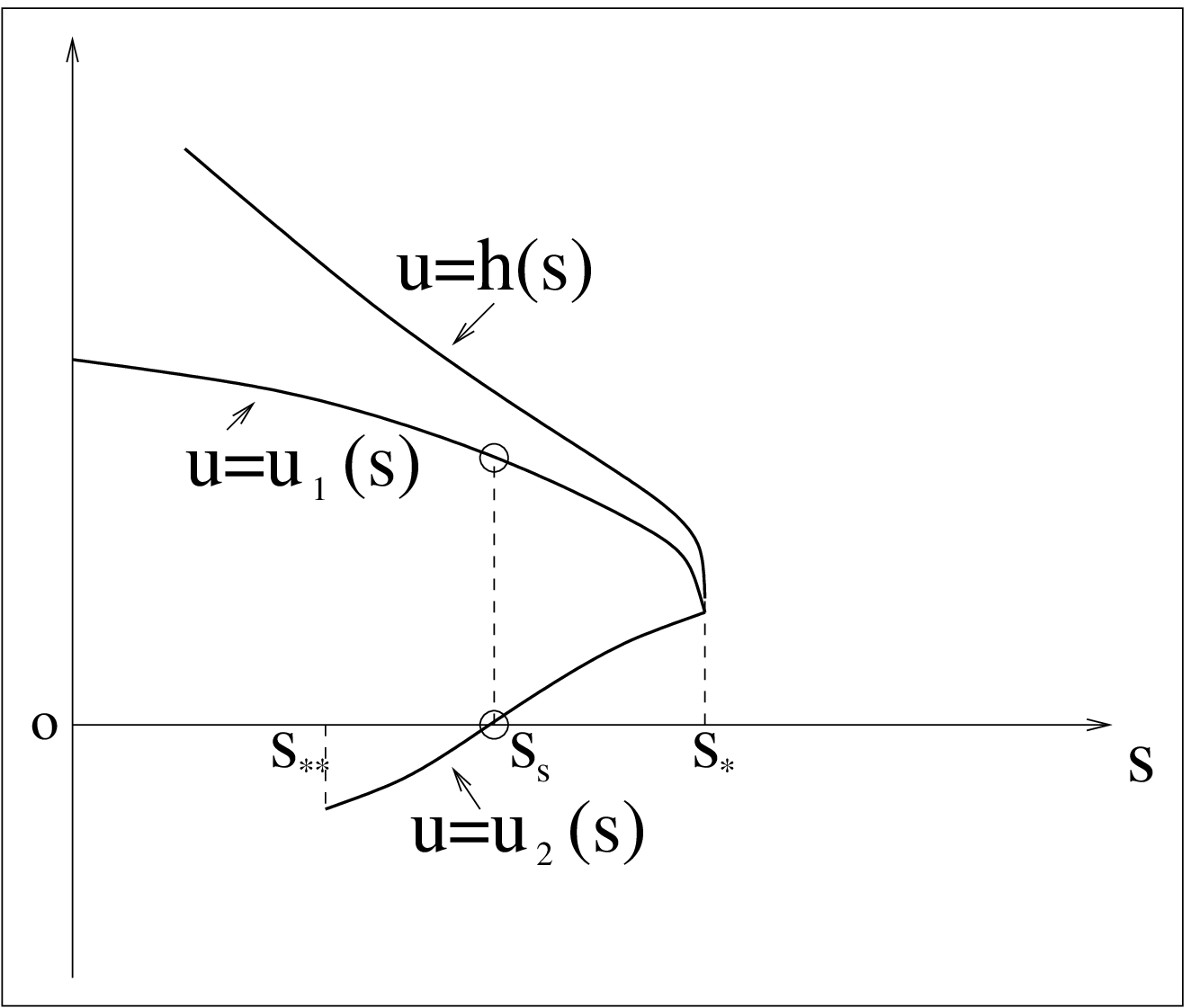}\quad\includegraphics[scale=0.43]{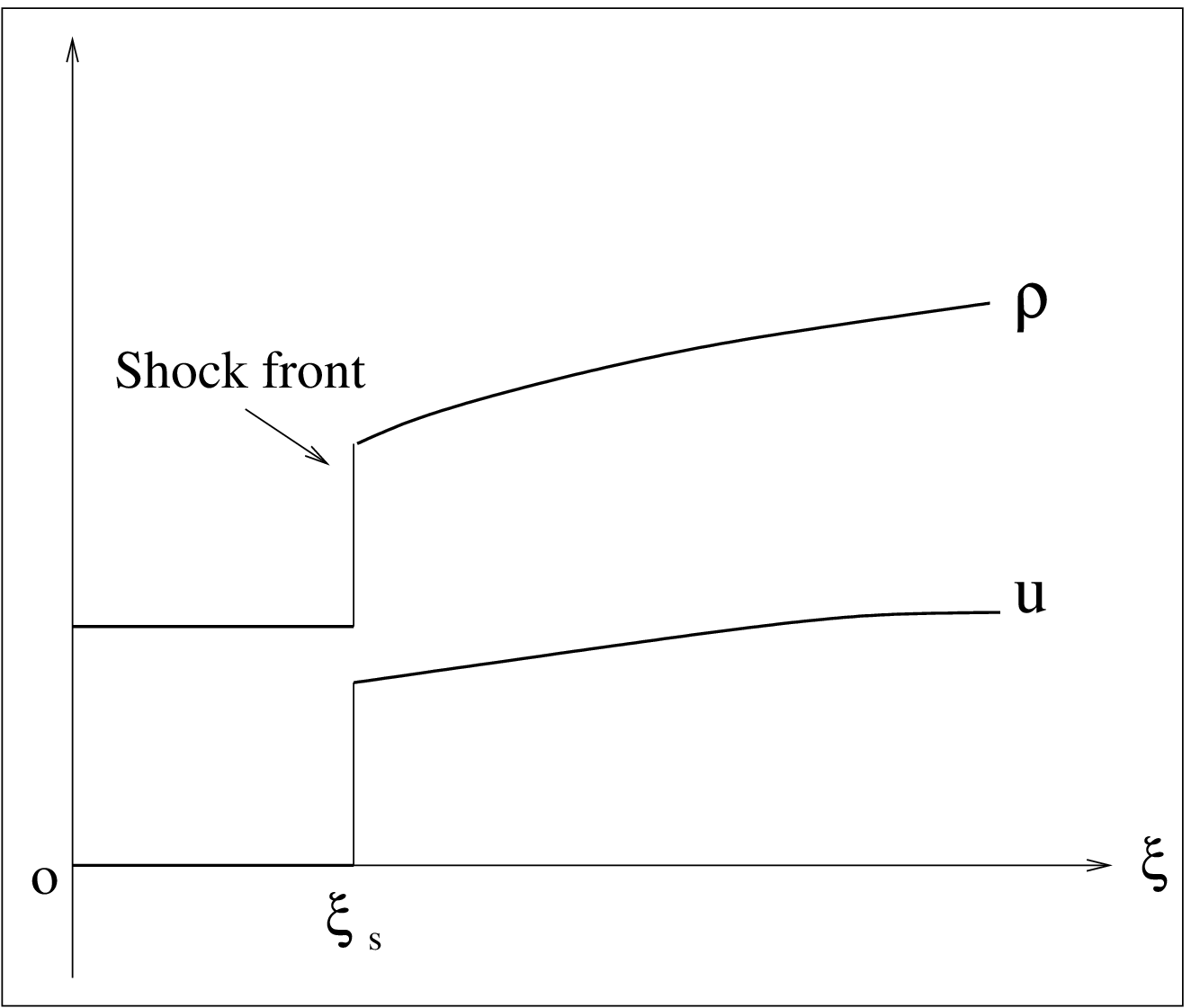}
\caption{\footnotesize Discontinuous solution with a single rarefaction shock.}
\label{fig4}
\end{center}
\end{figure}

Let $(u_2, \tau_2)(s)$ ($s_{**}\leq s<s_*$)  be determined by (\ref{backside}) and $\tau_2(s)>\tau_1(s)$.
It is easy to see that $u_2(s_*)=u_1(s_*)>0$ and
\begin{equation}\label{81701}
\tau_2(s_{**})=f(\tau_1(s_{**}))>\tau_2^i.
\end{equation}

\vskip 8pt
If $u_2(s_{**})\leq 0$ then there exists a $s_{s}\in [s_{**}, s_{*})$ such that $u_2(s_{s})=0$; see Figure \ref{fig4}(left). In this case, the self-similar solution of the problem (\ref{AE}), (\ref{IBV}) has the form:
$$
(u, \rho)(x, t)=\left\{
                 \begin{array}{ll}
                   (u_1, \rho_1)(s), & \hbox{$s<s_{s}$,} \\[4pt]
                   (0, \rho_2(s_{s})), & \hbox{$s>s_{s}$,}
                 \end{array}
               \right.
$$
where $s=t/x$; see Figure \ref{fig4}(right).


\vskip 8pt

If $u_2(s)>0$ as $s\in [s_{**}, s_{*}]$.
We then consider system (\ref{ODE3}) with the initial data
\begin{equation}\label{ID3}
(s, \rho)\mid_{u=u_2(s_{**})}~=~(s_{**}, \rho_2(s_{**})).
\end{equation}

\begin{lem}\label{100503}
When $\delta>0$ is sufficiently small, the initial value problem (\ref{ODE3}), (\ref{ID3}) has a solution $(\bar{s}, \bar{\rho})(u)$ on $(u_2(s_{**})-\delta, u_2(s_{**}))$. Moreover, this solution satisfies ${\rm d} \bar{s}/{\rm d} u<0$ and $\bar{s}^2 p'(\bar{\rho})-(1-u\bar{s})^2<0$ as $u\in\big(u_2(s_{**})-\delta, u_2(s_{**})\big)$.
\end{lem}
\begin{proof}
It is easy to see that the initial value problem is a classically well-posed problem which has a unique local solution.
From(\ref{1}), (\ref{2}), and (\ref{102502}) we have
\begin{equation}\label{122503}
\frac{1}{s_{**}}=u_2(s_{**})+\tau_2(s_{**})\sqrt{-p'(\tau_2(s_{**}))}.
\end{equation}
Hence, we have
$\bar{s}^2 p'(\bar{\rho})-(1-u\bar{s})^2=0$ as $u=u_2(s_{**})$.

From (\ref{81701}) and (\ref{122503}), we have
$$
\begin{aligned}
 &\frac{{\rm d}}{{\rm d} u}\big(\bar{s}^2 p'(\bar{\rho})-(1-u\bar{s})^2\big)\\=&
 \big(2\bar{s} p'(\bar{\rho})+2u(1-u\bar{s})\big)\cdot\Big(\frac{\bar{s}^2 p'(\bar{\rho})-(1-u\bar{s})^2}{2p'(\bar{\rho}) u\bar{s}}\Big)+\frac{2\bar{s}\bar{\tau}^3p''(\bar{\tau})}{p'(\bar{\rho})}(1-u\bar{s})>0
\end{aligned}
$$
as $u=u_2(s_{**})$.

Thus, when $\delta>0$ is sufficiently small we have
$$\bar{s}^2 p'(\bar{\rho})-(1-u\bar{s})^2<0\quad \mbox{as}~~ u\in \big(u_2(s_{**})-\delta, u_2(s_{**})\big).$$
We then complete the proof of this lemma.
\end{proof}

Let $u=\bar{u}_1(s)$ be the inverse function of $s=\bar{s}(u)$ and $\bar{\rho}_1(s)=\bar{\rho}(\bar{u}_1(s))$.
It is obviously that $(\bar{u}_1, \bar{\rho}_1)(s)$ satisfies the system (\ref{ODE2}) in $\big(s_{**}, \bar{s}(u_2(s_{**})-\delta)\big)$. Moreover, by Lemma \ref{100503} we also have
$$
0~<~\bar{u}_1(s)~<~h(\bar{\rho}_1(s), s)\quad \mbox{and}\quad \bar{\tau}_1(s)>\tau_2^i
$$
as $s\in \big(s_{**}, \bar{s}(u_2(s_{**})-\delta)\big)$; see Figure \ref{fig12}(left). Thus, when $s>\bar{s}(u_2(s_{**})-\delta)$ the discussion is similar to that of section 3.1. The structures of the solution can be illustrated  in Figure \ref{fig12}.

\begin{figure}[htbp]
\begin{center}
\includegraphics[scale=0.38]{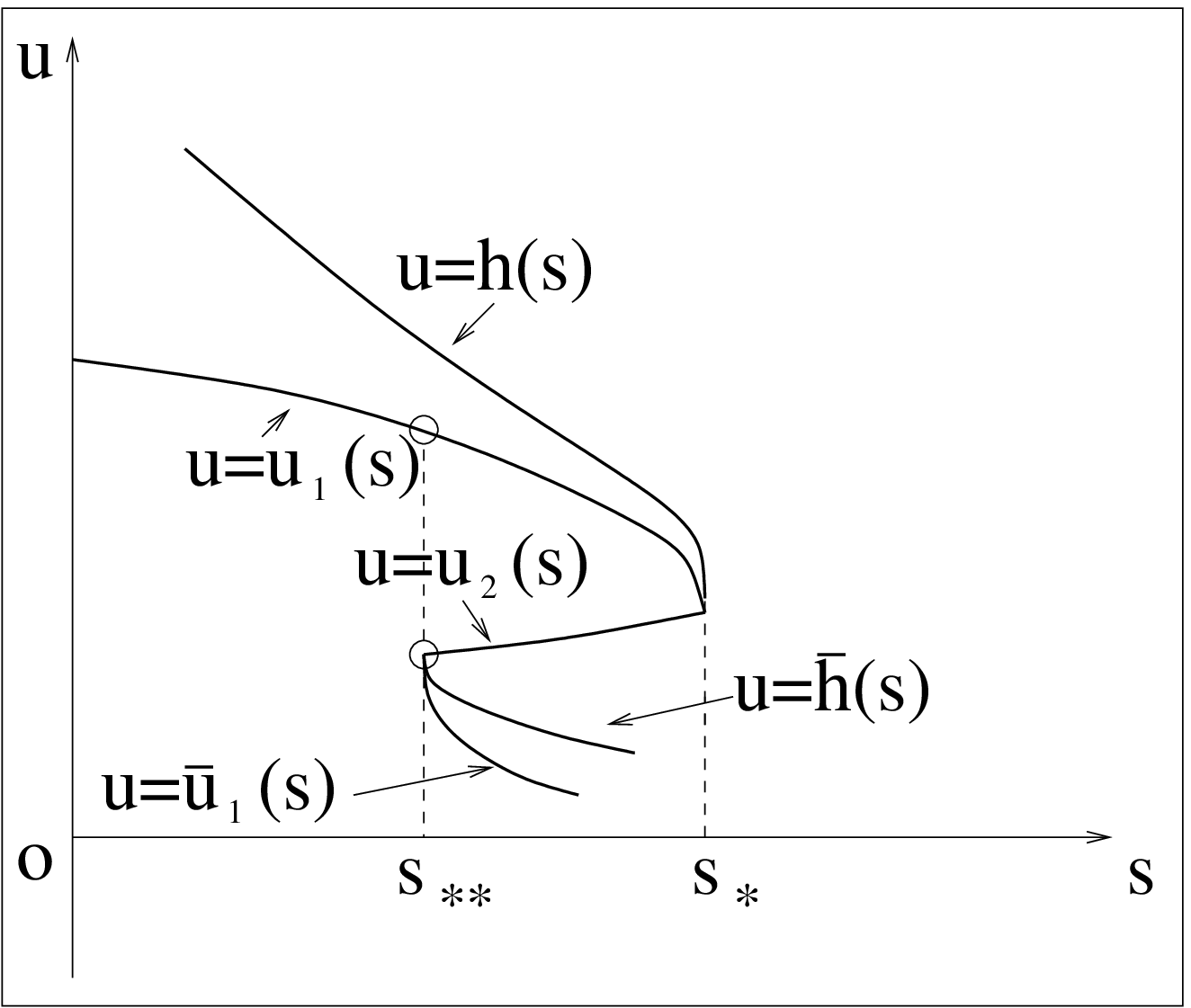}~\includegraphics[scale=0.38]{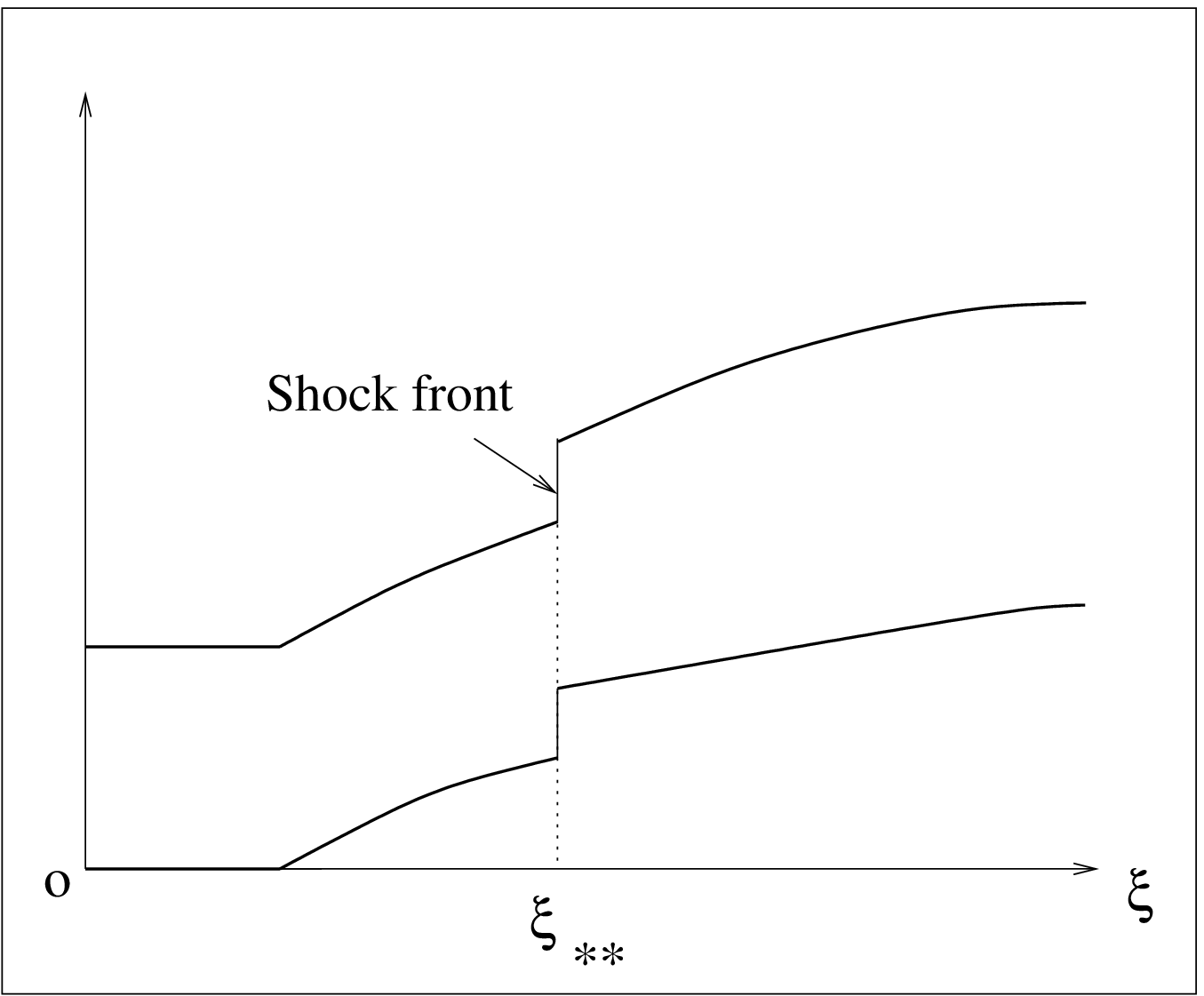}~
\includegraphics[scale=0.38]{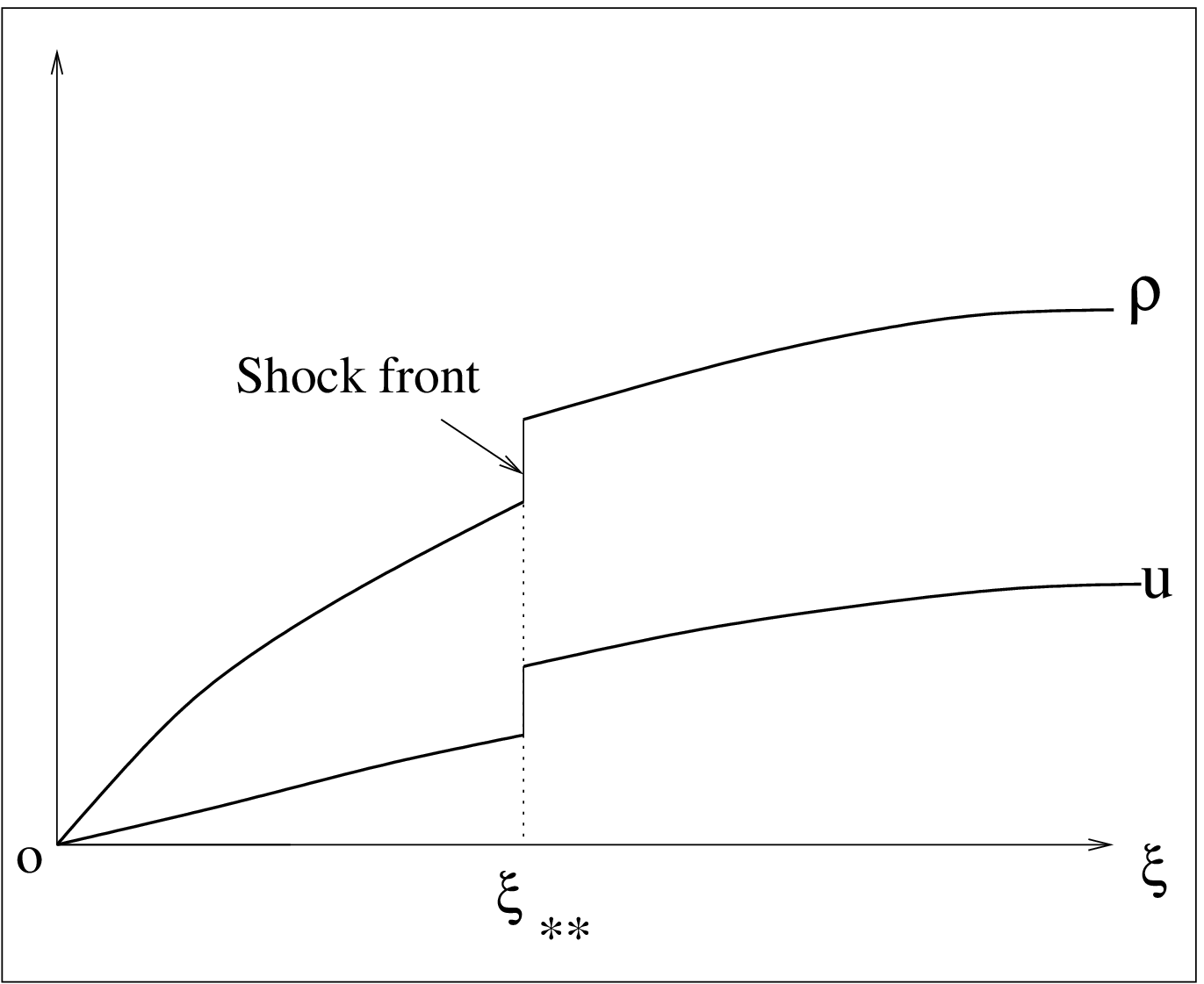}
\caption{\footnotesize Discontinuous solutions with a single rarefaction shock.}
\label{fig12}
\end{center}
\end{figure}

\subsubsection{\bf $\tau_0<\tau_1^i$} It is similar to $\tau_1^i\leq \tau_0<\tau_2^i$ that there have the four cases. We only discuss the forth case, i.e., there exists a $s_*>0$ such that $0<u_1(s)<h(\rho_1(s), s)$ as $0<s<s_*$ and $0<u_1(s_*)=h(\rho_1(s_*), s_*)<\frac{1}{s_{*}}$.

If $\tau_0\geq\hat{\tau}_1$, we can have Lemmas \ref{100502} and \ref{lem1701}, since $f(\tau_1(s))$ can be also defined for $s\in [\tau_0, s_*]$.
Then the discussion will be similar to that of section 3.2.2.  We omit the details.

If $\tau_0<\hat{\tau}_1$,
then we let $\hat{s}$ be the point such that $\tau_1(\hat{s})=\hat{\tau}_1$.
Since $u_1(s)<h(\rho_1(s), s)$ as $0<s<s_*$, we have $$\frac{1}{\hat{s}}>u_1(\hat{s})+\hat{\tau}_1\sqrt{-p'(\hat{\tau}_1)}=u_1(\hat{s})+\hat{\tau}_1\sqrt{-p'(\hat{\tau}_2)}
=u_1(\hat{s})+\hat{\tau}_1\sqrt{-p'(f(\hat{\tau}_1))},$$
since $f(\hat{\tau}_1)=\hat{\tau}_2$.
Then there exists a $s_{**}\in (\hat{s}, s_*)$ such that $$u_1(s)+\tau_1(s)\sqrt{-p'(\tau_1(s))}<\frac{1}{s}<u_1(s)+\tau_1(s)\sqrt{-p'(f(\tau_1(s)))}$$ as $s\in (s_{**}, s_*)$ and $1/s_{**}=u_1(s_{**})+\tau_1(s_{**})\sqrt{-p'(f(\tau_1(s_{**})))}$. Then the discussion will be similar to that of section 3.2.2. We omit the details.

\subsection{Equation of state III}
We first define
\begin{equation}\label{b2}
b_1:=\lim\limits_{\rho\rightarrow\tilde{\rho}_1^{+}}\sqrt{p'(\rho)}, \quad b_2:=\lim\limits_{\rho\rightarrow\tilde{\rho}_2^{-}}\sqrt{p'(\rho)},
\end{equation}
where $\tilde{\rho}_i=\frac{1}{\tilde{\tau}^i}$ ($i=1, 2$).
\subsubsection{$\tau_0\geq\tilde{\tau}_2$}
The discussion is similar to that of section 3.1, since  $\tau_1'(s)>0$ as $s>0$  and $p''(\tau)>0$ as $\tau<\tau_0$.
\subsubsection{$\tilde{\tau}_1<\tau_0<\tilde{\tau}_2$}
Let $s_{*}$ be defined so that
$$
\int_{\rho_0}^{\tilde{\rho}_2}\frac{1}{\rho}{\rm d}\rho~=~\int_{0}^{s_{*}}\frac{2u_0}{u_0 s-1} {\rm d}s.
$$
Hence, we have
$$
u_1(s)=u_0, \quad \rho_1(s)=\rho_0\exp\Big(\int_{0}^{s}\frac{2u_0}{u_0 s-1} {\rm d}s\Big), \quad 0<s<s_*.
$$

If $u_0\leq 1/s_*-b_2$ then the discussion for $s\geq s_*$ is similar to that of section 3.1,
i.e., the problem (\ref{AE}), (\ref{IBV}) has a continuous solution.

In what follows, we are going to discuss the case of $u_0> 1/s_*-b_2$. It is similar to the fourth case of section 3.2.2 that the problem does not have a compression shock wave solution. So, we look for a rarefaction shock wave solution.

For $\tilde{\tau}_1\leq\tau_1\leq\tilde{\tau}_2$, we let $g(\tau_1)$ be defined such that
\begin{equation}\label{g1}
\frac{p(\tau_1)-p(g(\tau_1))}{\tau_1-g(\tau_1)}=p'(g(\tau_1)) \quad \mbox{and}\quad g(\tau_1)>\tilde{\tau}_2.
\end{equation}

\begin{lem}
There  exist a $s_{**}\in (0, s_*)$ such that for any
$s\in [s_{**}, s_{*}]$,
there exists an admissible forward rarefaction shock with the speed $1/s$ and the front side state $(u_1, \rho_1)(s)$.
\end{lem}
\begin{proof}
According to Corollary \ref{cor5},
in order that $(u_1, \rho_1)(s)$ can be the front side state of an admissible forward rarefaction shock with the speed $1/s$, there holds
$$
0<\frac{1}{s}~\leq~ u_1(s)+\tau_1(s)\sqrt{-p'(g(\tau_1(s)))}.
$$

It follows from $u_0>1/s_*-b_2$ that
$$
\frac{1}{s_*}~<~u_0+b_2=u_0+\tau_1(s_*)\sqrt{-p'(g(\tau_1(s_*)))},
$$
since $g(\tau_1(s_*))=\tau_1(s_*)=\tilde{\tau}_2$.
Therefore, there exists a $s_{**}\in (0, s_*)$ such that $1/s<u_0+\tau_1(s)\sqrt{-p'(g(\tau_1(s)))}$ as $s\in (s_{**}, s_*)$ and
$
1/s_{**}=u_0+\tau_1(s_{**})\sqrt{-p'(g(\tau_1(s_{**})))}.
$
 We then complete the proof of this lemma.
\end{proof}

 Let $(u_2, \tau_2)(s)$ ($s_{**}\leq s<s_*$)  be determined by (\ref{backside}) and $\tau_2(s)>\tau_1(s)$. It is easy to see that $u_2(s_*)=u_0>0$. Then, the discussion will be similar to the forth case of section 3.2.2. We omit the details.

\subsubsection{$\tau_0<\tilde{\tau}_1$} We have the following two cases:
\begin{itemize}
  \item There exists a $s_*>0$ such that $u_1(s)<h(\rho_1(s), s)$ as $s<s_*$ and $u_1(s_*)=h(\rho_1(s_*), s_*)=0$ and $\rho_1(s_*)<\tilde{\rho}_1$.
  \item There exists a $s_1>0$ such that $0<u_1(s_1)<h(\rho_1(s), s)<1/s$ as $0<s<s_1$ and $\rho_1(s_1)=\tilde{\rho}_1$.
\end{itemize}
The structure of the solution for the first case can be illustrated  by Figure \ref{fig2}. We only need to discuss the second case.

Let $s_{*}$ be defined so that
$$
\int_{\tilde{\rho}_1}^{\tilde{\rho}_2}\frac{1}{\rho}{\rm d}\rho~=~\int_{s_1}^{s_{*}}\frac{2u_1(s_1)}{u_1(s_1) s-1} {\rm d}s.
$$
Hence, we have
$$
u_1(s)= u_1(s_1), \quad \rho(s)=\tilde{\rho}_1\exp\Big(\int_{s_1}^{s}\frac{2u_1(s_1)}{u_1(s_1) s-1} {\rm d}s\Big), \quad u_1(s)<\frac{1}{s}=h(\rho_1(s), s)\quad\mbox{as}\quad s_1<s<s_*.
$$

\begin{figure}[htbp]
\begin{center}
\includegraphics[scale=0.36]{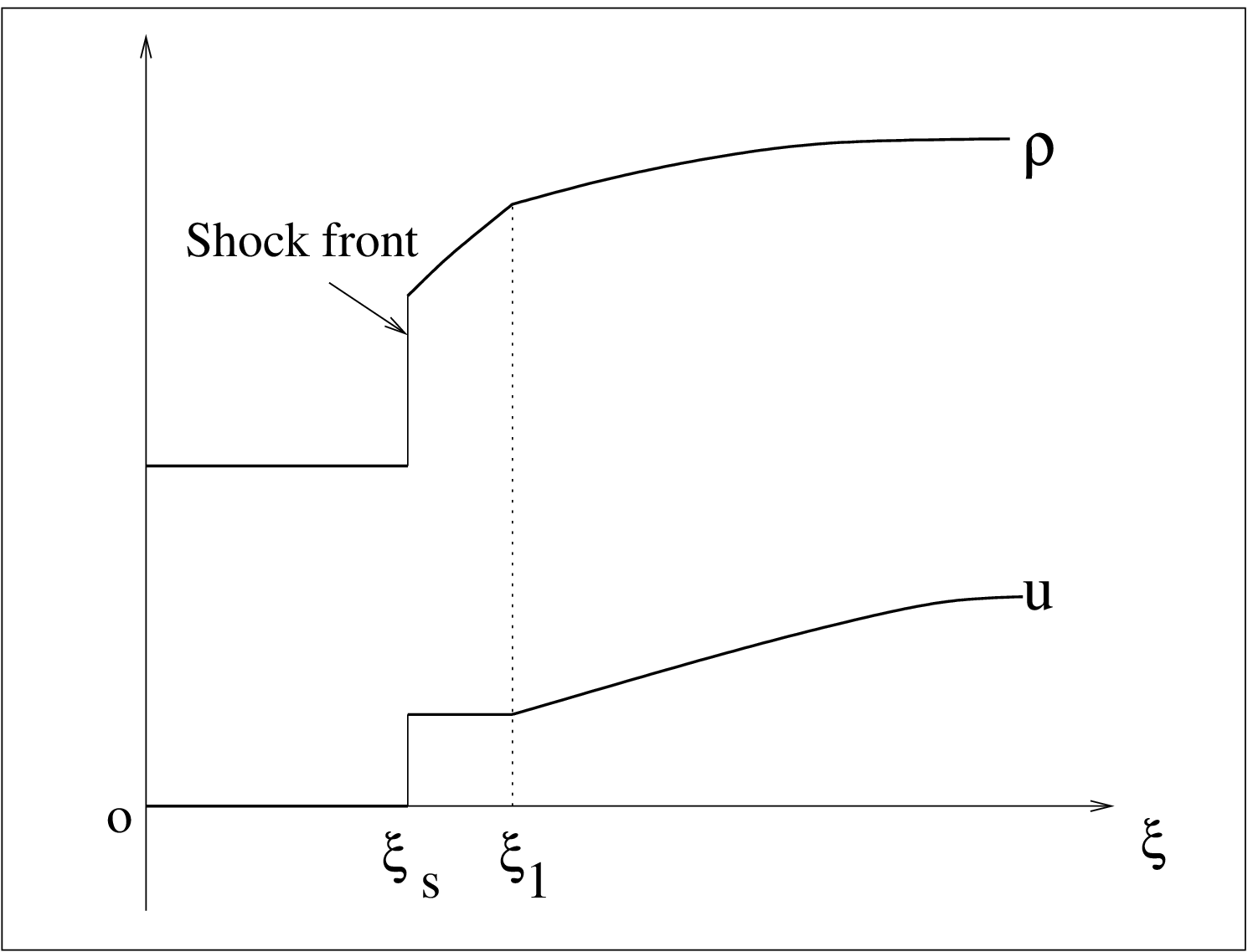}~\includegraphics[scale=0.36]{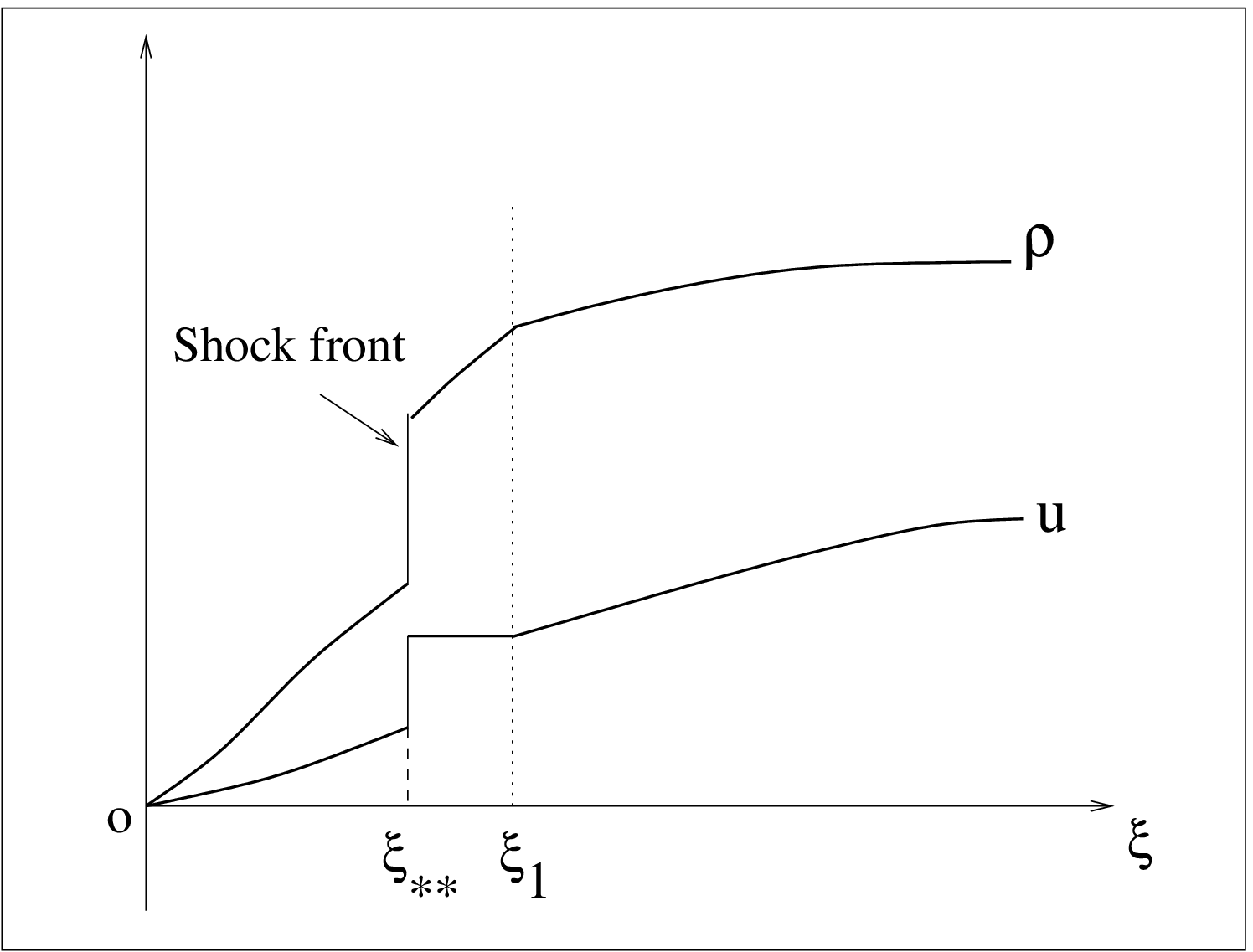}~ \includegraphics[scale=0.36]{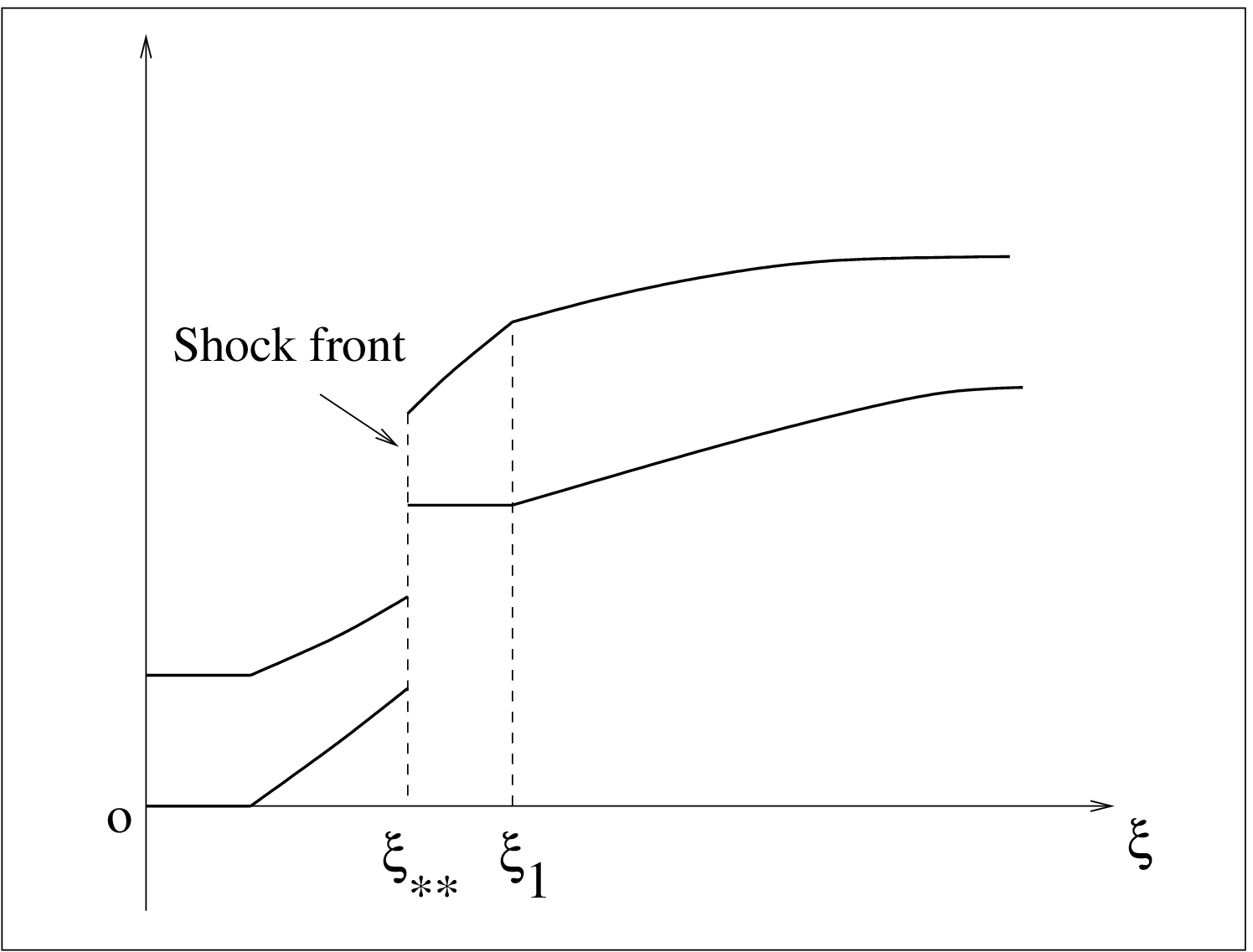}
\caption{\footnotesize Discontinuous solutions with a single rarefaction shock.}
\label{fig5}
\end{center}
\end{figure}

If $u_1(s_1)\leq 1/s_*-b_2$ then the discussion for $s>s_*$ is similar to that of section 3.1.
In what follows, we are going to discuss the case of $u_1(s_1)> 1/s_*-b_2$. We look for a rarefaction shock wave solution.

\begin{lem}
There exists a $s_{**}\in (s_1, s_*)$ such that for any
$s\in [s_{**}, s_{*}]$,
there exists an admissible forward rarefaction shock with the speed $1/s$ and the front side state $(u_1, \rho_1)(s)$.
\end{lem}
\begin{proof}
Since $u_1(s)<h(\rho_1(s), s)$ as $0<s<s_*$, we have
$$
\frac{1}{s}>u_1(s)+\tau_1(s)\sqrt{-p'(\tau_1(s))}\quad \mbox{as}\quad 0<s<s_*.
$$
From assumption (A2) we also have
\begin{equation}\label{122601}
\frac{1}{s_1}>u_1(s_1)+b_1>u_1(s_1)+\tau_1(s_1)\sqrt{-p'(g(\tau_1(s_1)))}.
\end{equation}
It follows from $u_1(s_1)> 1/s_*-b_2$ that
\begin{equation}\label{12201}
\frac{1}{s_*}~<~u_1(s_1)+b_2=u_1+\tau_1(s_*)\sqrt{-p'(g(\tau_1(s_*)))}.
\end{equation}
Combining with (\ref{122601}) and (\ref{12201}), there exists a $s_{**}\in (s_1, s_*)$ such that $1/s<u_1(s)+\tau_1(s)\sqrt{-p'(g(\tau_1(s)))}$ as $s\in (s_{**}, s_*)$ and $1/s_{**}=u_1(s_{**})+\tau_1(s_{**})\sqrt{-p'(g(\tau_1(s_{**})))}$. Thus by Corollary \ref{cor5}
we have this lemma.
\end{proof}
Hence, the discussion for this case will be similar to the forth case of section 3.2.2. The wave structures of the solution can be illustrated in Figure \ref{fig5}.

\section{\bf Self-similar solutions for $u_0<0$}
In this section, we will construct the self-similar solutions of the problem (\ref{AE}), (\ref{IBV})  for $u_0>0$.

\subsection{Equation of state I}
From $u_0<0$ we have
\begin{equation}\label{102601}
\frac{{\rm d} u_1}{{\rm d} s}>0\quad\mbox{and}\quad \frac{{\rm d} \rho_1}{{\rm d} s}>0.
\end{equation}
\begin{lem}\label{lem4.1}
For any $u_0<0$ there exists a $s_{*}>0$ such that $u_1(s)<h(\rho_1(s), s)$ as $0<s<s_*$ and $u_1(s_*)=h(\rho_1(s_*), s_*)<0$; see Figure \ref{fig6}(left).
\end{lem}
\begin{proof}
The proof of this lemma proceeds in two steps.

{\bf Step 1.} We first claim that the integral curves $u=u_1(s)$ and $u=h(\rho_1(s), s)$ can not intersect at the $s-$axis.

We argue by contradiction.
Suppose  there is a $s_1>0$ such that $u_1(s)<0<h(\rho_1(s), s)$ as $s<s_1$ and  $u_1(s_1)=h(\rho_1(s_1), s_1)=0$.
Then we have $s_1\sqrt{p'(\rho_1(s_1))}=1$.
We consider the following ordinary system
\begin{equation}\label{100601}
\left\{
 \begin{aligned}
    &\frac{{\rm d} u}{{\rm d} r}=2p'(\rho) us,\\
    &\frac{{\rm d} s}{{\rm d} r}= s^2 p'(\rho)-(1-us)^2,\\
& \frac{{\rm d} \rho}{{\rm d} r}=2\rho u (1-us)
  \end{aligned}
\right.
\end{equation}
At the point $(u, s, \rho)=(0, s_1, \rho_1(s_1))$, we find the linear part of the right-hand side of (\ref{100601}) is given by $M(u, s-s_1, \rho-\rho_1(s_1))^{T}$ where
$$
M=\left(
    \begin{array}{ccc}
      \displaystyle 2\sqrt{p'(\rho_1(s_1))} & 0 & 0\\[6pt]
      \displaystyle \frac{2}{\sqrt{p'(\rho_1(s_1))}} & 2\sqrt{p'(\rho_1(s_1))} & \displaystyle\frac{p''(\rho_1(s_1))}{p'(\rho_1(s_1))} \\[6pt]
\displaystyle\\ 2\rho_1(s_1) &0  &0
    \end{array}
  \right).
$$
Since
$$
\begin{aligned}
&\frac{p''(\rho_1(s_1))}{p'(\rho_1(s_1))}\cdot\frac{\rho_1(s_1)}{\sqrt{p'(\rho_1(s_1))} }+\frac{2}{\sqrt{p'(\rho_1(s_1))}}\\
&\qquad =~\frac{1}{(p'(\rho_1(s_1)))^{3/2}}\Big(\rho_1(s_1) p''(\rho_1(s_1))+2p'(\rho_1(s_1))\Big)~=~\frac{\tau_1^3(s_1)p''(\tau_1(s_1))}{(p'(\rho_1(s_1)))^{3/2}}~>~0,
\end{aligned}
$$
we have that along the integral curves of (\ref{100601}), $ \frac{{\rm d} s}{{\rm d} u}\rightarrow -\infty$ as $(u, s, \rho)\rightarrow (0, s_1, \rho_1(s_1))$; see Figure \ref{fig13}.
This leads to a contradiction.
Thus the integral curves $u=u_1(s)$ and $u=h(\rho_1(s), s)$ can not intersect at the $s-$axis.

\begin{figure}[htbp]
\begin{center}
\includegraphics[scale=0.45]{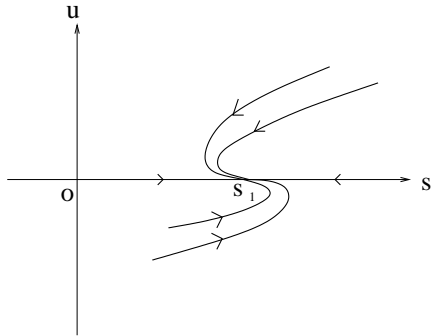}
\caption{\footnotesize  Integration curves of (\ref{100601}).}
\label{fig13}
\end{center}
\end{figure}


{\bf Step 2.} Let $m=\inf\limits_{\rho\in[\rho_0, +\infty)}\sqrt{p'(\rho)}$.
Suppose that the curves $u=u_1(s)$ and $u=h(\rho_1(s), s)$ do not intersect with each other.
Then by (\ref{102601}) we have $$\lim\limits_{s\rightarrow +\infty}u_1(s)=u_{\infty}<-m.$$
By (\ref{ODE2}), we have
$$
\frac{{\rm d} u_1}{{\rm d} s}=\frac{p'(\rho_1) u_1s}{s^2 p'(\rho_1)-(1-u_1s)^2}>\frac{m^3s}{(1-u_{0}s)^2}.
$$
Hence,
$$
u_{\infty}-u_0~>~u_1(s)-u_0 ~>~ \int_{0}^{s} \frac{m^3 s}{(1-u_{0}s)^2}  {\rm d}s \quad \mbox{as}\quad s>0,
$$
which leads to a contradiction.
We then have this lemma.
\end{proof}

Lemma \ref{lem4.1} implies that if $u_0<0$ then the problem (\ref{AE}), (\ref{IBV}) does not have a global continuous solution. So, we need to look for a shock wave solution.
\begin{lem}
For any $s\in (0, s_*)$ there exists an admissible compression shock with the speed $1/s$ and the front side state $(u_1, \rho_1)(s)$.
\end{lem}
\begin{proof}
This lemma can be proved by the fact that
$
1/s~>~u_1(s)+\sqrt{p'(\rho_1(s))}
$
as $0<s<s_*$.
\end{proof}
Let the back side state of the shock $(u_2, \rho_2)(s)$ ($0<s\leq s_*$) be determined by (\ref{backside}). It is easy to see that
$$
u_2(s_{*})=u_1(s_*)<0\quad \mbox{and} \quad \lim\limits_{s\rightarrow 0}u_2(s)=+\infty.
$$
Therefore, there exists a $s_{s}\in (0, s_{*})$ such that $u_2(s_{s})=0$. Hence, the self-similar solution of the problem (\ref{AE}), (\ref{IBV}) has the form
$$
(u, \rho)(x, t)=\left\{
                 \begin{array}{ll}
                   (u_1, \rho_1)(s), & \hbox{$s<s_{s}$,} \\[4pt]
                   (0, \rho_2(s_{s})), & \hbox{$s>s_{s}$,}
                 \end{array}
               \right.
$$
where $s=t/x$; see Figure \ref{fig6}(right).

\begin{figure}[htbp]
\begin{center}
\includegraphics[scale=0.35]{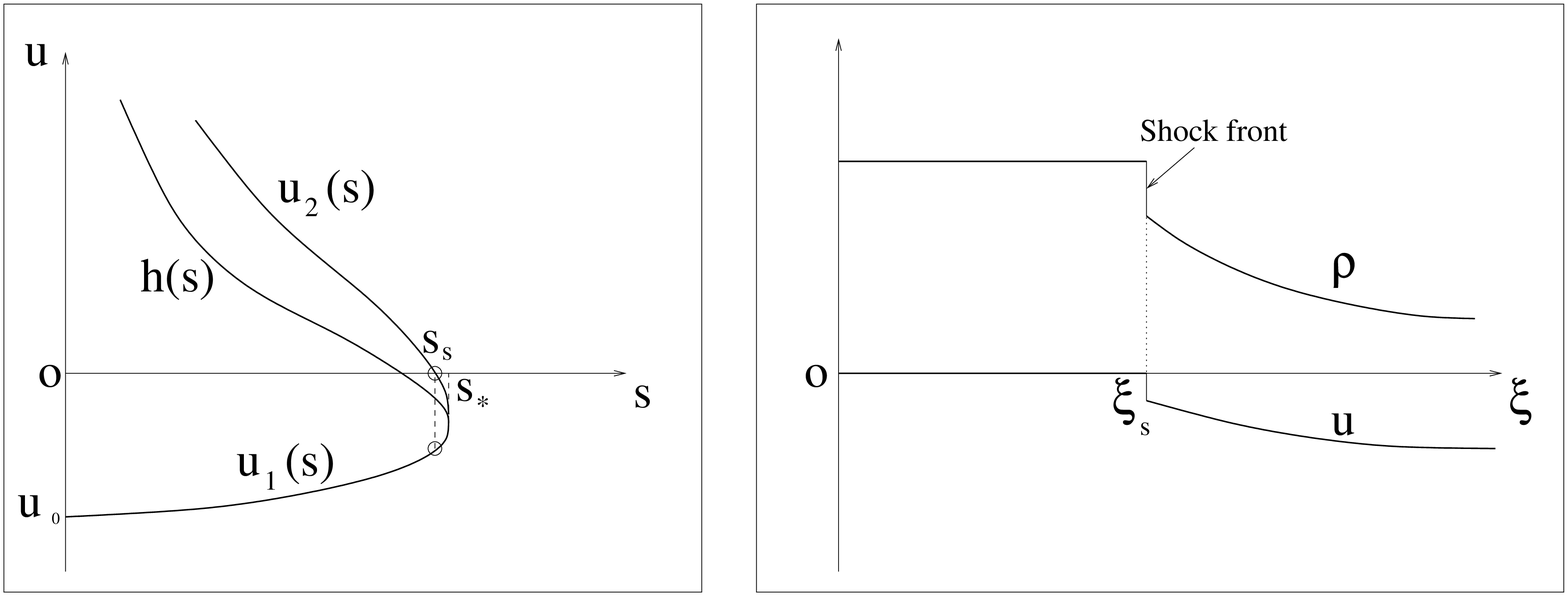}
\caption{\footnotesize Discontinuous solution with a single compression shock.}
\label{fig6}
\end{center}
\end{figure}

\subsection{Equation of state II}
\subsubsection{$\tau_0\leq \tau_1^{i}$} The discussion is similar to that of section 4.1, since  $\tau_1'(s)<0$ as $s>0$ and $p''(\tau)>0$ as $\tau<\tau_0$.

\subsubsection{$\tau_1^{i}<\tau_0< \tau_2^{i}$}
There are the following two cases:
\begin{itemize}
  \item There exists a $s_*>0$
such that $u_1(s)<h(\rho_1(s), s)$ as $0<s<s_*$ and $u_1(s_*)=h(\rho_1(s_*), s_*)<0$.
  \item There exists a $s_*>0$
such that $u_1(s)<h(\rho_1(s), s)$ as $0<s<s_*$ and $u_1(s_*)=h(\rho_1(s_*), s_*)=0$. (By Lemma \ref{lem4.1}, we have $\tau_1(s_*)\in (\tau_1^{i}, \tau_2^i)$ in this case.)
\end{itemize}

The structure of the solution for the first case is similar to that of section 4.1, since $
1/s~>~u_1(s)+\sqrt{p'(\rho_1(s))}
$
as $0<s<s_*$.

The solution for the second case has the form
$$
(u, \rho)(x, t)=\left\{
                 \begin{array}{ll}
                   (u_1, \rho_1)(s), & \hbox{$s<s_{*}$,} \\[4pt]
                   (0, \rho_1(s_{*})), & \hbox{$s>s_{*}$.}
                 \end{array}
               \right.
$$



\subsubsection{$\tau_0> \tau_2^{i}$} There are the following two cases:
\begin{itemize}
  \item There exists a $s_{*}>0$ such that $u_1(s)<0<h_1(\rho_1(s), s)$ as $s<s_*$ and  $u_1(s_*)=h(\rho_1(s_*), s_*)=0$  and $\tau_1(s_*)\in (\tau_1^{i}, \tau_2^i)$.
\item There exists a $s_{*}>0$ such that $u_1(s)<h_1(\rho_1(s), s)$ as $s<s_*$ and $u_1(s_*)=h(\rho_1(s_*), s_*)<0$ and $\tau_1(s_*)\in[\tau_2^{i}, \tau_0)\cup(0, \tau_1^{i}]$.
\end{itemize}
\begin{rem}
By (\ref{100504}) and (\ref{100505}),  it is impossible to have a $s_*>0$ such that $u_1(s)<h_1(\rho_1(s), s)$ as $s<s_*$ and $u_1(s_*)=h(\rho_1(s_*), s_*)<0$  and $\tau_1(s_*)\in  (\tau_1^{i}, \tau_2^i)$.
\end{rem}

In what follows,  we are going to discuss the second case. 
For $\tau_2^i<\tau_1<\tau_3$,
we let $\psi(\tau_1)$ be defined such that
\begin{equation}
\frac{p(\tau_1)-p(\psi(\tau_1))}{\tau_1-\psi(\tau_1)}=p'(\psi(\tau_1)) \quad \mbox{and}\quad\tau_1^i< \psi(\tau_1)<\tau_2^i.
\end{equation}
Here, $\tau_3$ is defined in Corollary \ref{cor2}.

\vskip 4pt
We first consider the case of $\tau_2^{i}<\tau_1(s_*)<\tau_0<\tau_3$.
Let
\begin{equation}\label{100201}
\hat{\xi}(s)=u_1(s)+\tau_1(s)\sqrt{-p'(\psi(\tau_1(s)))}\quad \mbox{and}\quad F(s)=\frac{1}{s}-\hat{\xi}(s).
\end{equation}
Then we have
\begin{equation}\label{1006021}
\lim\limits_{s\rightarrow0}F(s)=+\infty
\end{equation}
and
\begin{equation}\label{1006022}
\begin{aligned}
F(s_*)&=\frac{1}{s_*}-u_1(s_*)-\tau_1(s_*)\sqrt{-p'(\psi(\tau_1(s_*)))}\\
&=\frac{1}{s_*}-u_1(s_*)-\tau_1(s_*)\sqrt{-p'(\tau_1(s_*))}+\tau_1(s_*)\Big(\sqrt{-p'(\tau_1(s_*))}-\sqrt{-p'(\psi(\tau_1(s_*)))}\Big)
\\
&=\tau_1(s_*)\Big(\sqrt{-p'(\tau_1(s_*))}-\sqrt{-p'(\psi(\tau_1(s_*)))}\Big)<0.
\end{aligned}
\end{equation}

Since $
1/s~>~u_1(s)+\sqrt{p'(\rho_1(s))}
$ and $\tau_2^i<\tau_1(s)<\tau_0$
as $0<s<s_*$, for any $s\in (0, s_*)$ there exists an admissible compression shock with the speed $1/s$ and the front side state $(u_1, \rho_1)(s)$. Let the back side state of the shock $(u_2, \rho_2)(s)$ ($0<s\leq s_*$) be determined by (\ref{backside}). We have
\begin{equation}\label{100603}
u_2(s_{*})=u_1(s_*)<0\quad \mbox{and}\quad  \lim\limits_{s\rightarrow 0}u_2(s)=+\infty.
\end{equation}

From Corollary \ref{cor2} we know that
 if $F(s)=0$ then (\ref{backside}) has two solutions $(u_{2}^{+}, \tau_{2}^{+})(s)$ and $(u_{2}^{-}, \tau_{2}^{-})(s)$, where $u_{2}^{+}(s)>u_{2}^{-}(s)$ and $\tau_{2}^{+}(s)<\tau_1^i<\tau_{2}^{-}(s)<\tau_2^i$.
So, by (\ref{1006021}) and (\ref{1006022}) we can see that $u_2(s)$ is piecewise continuous on $(0, s_*)$.
Hence, we can not determine whether or not $u_2(s)$ has a zero point in $(0, s_*)$.

If there exists a $s_{s}\in (0, s_*)$ such that $u_2(s_s)=0$ then the problem (\ref{AE}), (\ref{IBV}) admits a discontinuous solution with a single shock with the speed $1/s_{s}$; see Figure \ref{fig6}(right).

\begin{figure}[htbp]
\begin{center}
\includegraphics[scale=0.43]{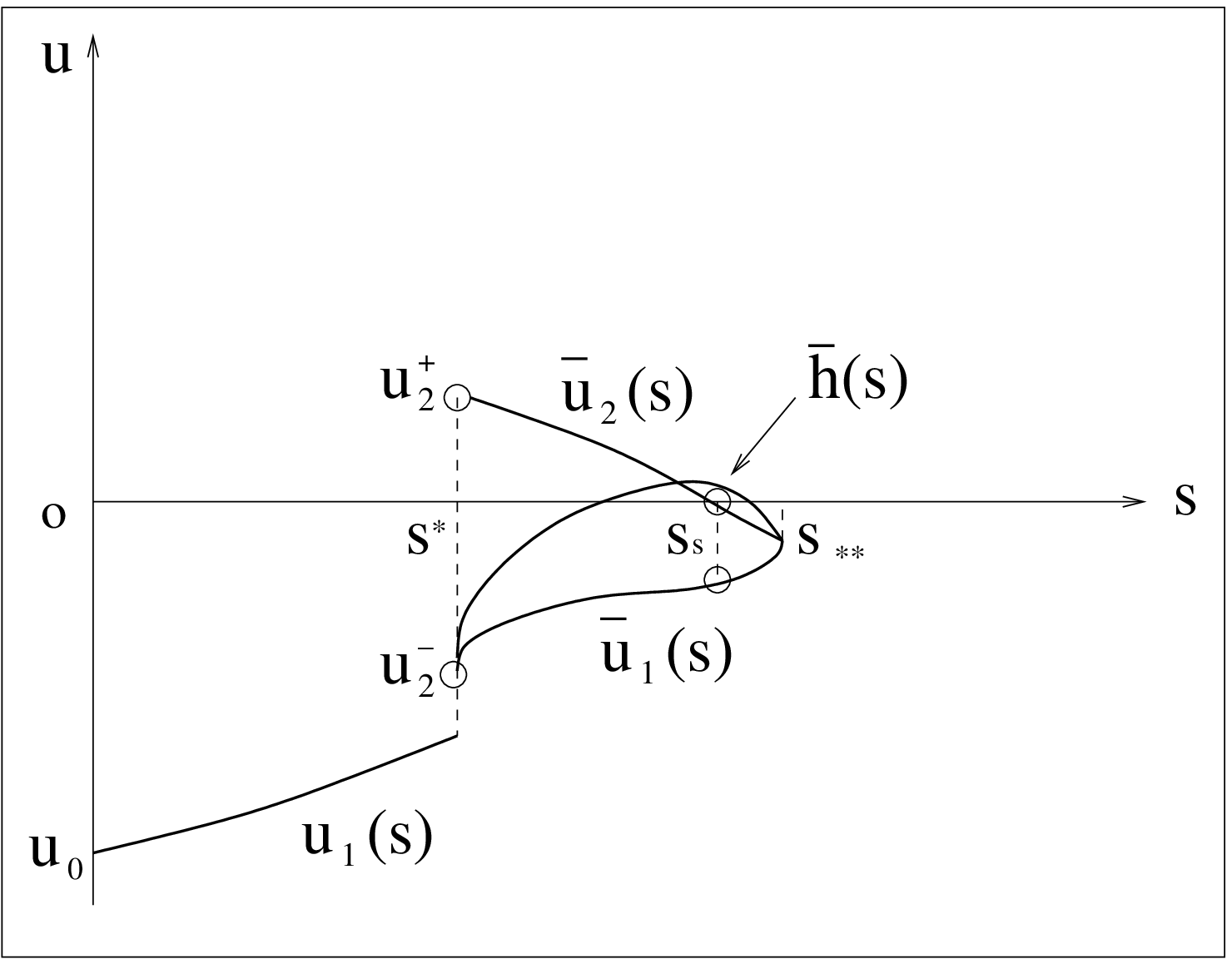}\quad\includegraphics[scale=0.43]{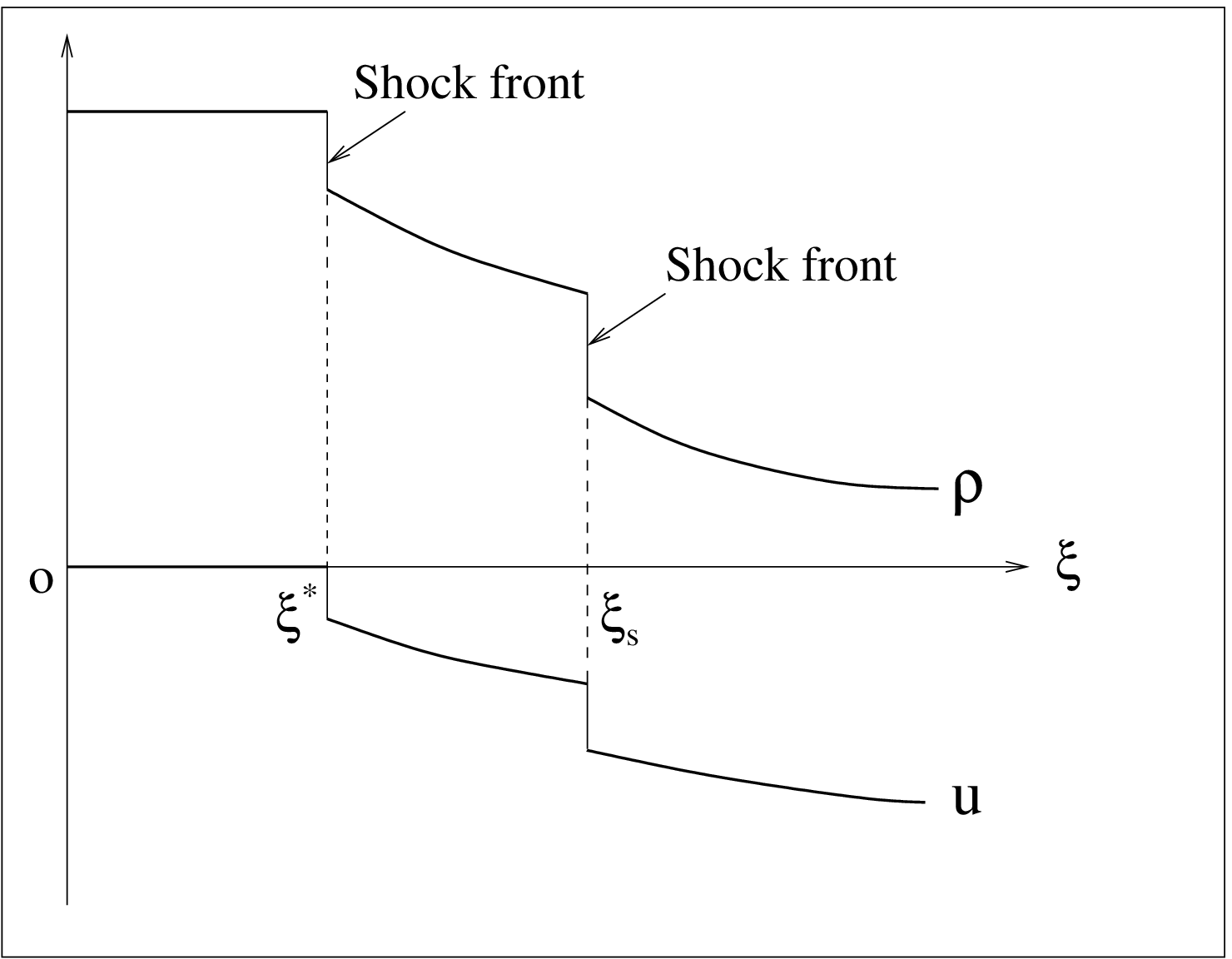}
\caption{\footnotesize Discontinuous solution with two compression shocks. }
\label{Figu2}
\end{center}
\end{figure}

If $u_2(s)\neq0$ for all $s\in(0, s_*)$, then by (\ref{100603}) there must exists a $s^{*}\in (0, s_*)$ such that
$$
F(s^{*})=0, \quad u_2^{+}(s^{*})>0, \quad u_2^{-}(s^{*})< 0, \quad \mbox{and}\quad
\tau_1^i<\tau_2^{-}(s^{*})<\tau_2^i.
$$
Then we consider system (\ref{ODE3})
with the data
\begin{equation}\label{100302}
(s, \rho)\mid_{u=u_2^{-}(s^{*})}~=~(s_*, \rho_2^{-}(s^{*})).
\end{equation}
\begin{lem}
When $\delta>0$ is sufficiently small, the initial value problem (\ref{ODE3}), (\ref{100302}) has a solution $(\bar{s}, \bar{\rho})(u)$ on $(u_2^{-}(s^{*}), u_2^{-}(s^{*})+\delta)$. Moreover, this solution satisfies $\frac{{\rm d} s}{{\rm d} u}>0$ and $s^2 p'(\rho)-(1-us)^2<0$ in $(u_2^{-}(s^{*}), u_2^{-}(s^{*})+\delta)$.
\end{lem}
\begin{proof}
The proof is similar to that of Lemma \ref{100503}, we omit the details.
\end{proof}
Let $u=\bar{u}_1(s)$ be the inverse function of $s=\bar{s}(u)$ and $\bar{\rho}_1(s)=\bar{\rho}(\bar{u}_1(s))$.
It is obviously that $(\bar{u}_1, \bar{\rho}_1)(s)$ satisfies (\ref{ODE2}) in $\big(s^{*}, \bar{s}(u_2^{-}(s^{*})+\delta)\big)$. Moreover, by Lemma \ref{100503} we also have
$$
0~<~\bar{u}_1(s)~<~h(\bar{\rho}_1(s), s)\quad \mbox{and}\quad \bar{\tau}_1(s)<\tau_2^i
$$
as $s\in\big(s^{*}, \bar{s}(u_2^{-}(s^{*})+\delta)\big)$. When $s>s^{*}$ there are two cases: (a)
 there exists a $s_{**}>s^{*}$
such that $\bar{u}_1(s)<h(\bar{\rho}_1(s), s)$ as $s^*<s<s_{**}$ and $\bar{u}_1(s_{**})=h(\bar{\rho}_1(s_{**}), s_{**})=0$; (b)
 there exists a $s_{**}>s^{*}$
such that $\bar{u}_1(s)<h(\bar{\rho}_1(s), s)$ as $s^*<s<s_{**}$ and $\bar{u}_1(s_{**})=h(\bar{\rho}_1(s_{**}), s_{**})<0$.

The solution for case (a) has the form
$$
(u, \rho)(x, t)=\left\{
                 \begin{array}{ll}
                   (u_1, \rho_1)(s), & \hbox{$s<s^{*}$,} \\[4pt]
                   (\bar{u}_1, \bar{\rho}_1)(s), & \hbox{$s^{*}<s<s_{**}$,}\\[4pt]
(0, \bar{\rho}_1(s_{**})), & \hbox{$s>s_{**}$.}
                 \end{array}
               \right.
$$
For case (b),
since $
1/s~>~\bar{u}_1(s)+\sqrt{p'(\bar{\rho}_1(s))}
$ and $\bar{\tau}_1(s)<\tau_2^i$
as $s^*<s<s_{**}$, for any $s\in (s^{*}, s_{**})$ there exists an admissible compression shock with the speed $1/s$ and the front side state $(\bar{u}_1, \bar{\rho}_1)(s)$.
Let the back side state of the shock $(\bar{u}_2, \bar{\rho}_2)(s)$ ($s^{*}<s\leq s_{**}$) be determined by (\ref{backside}).
Then we have
$
\bar{u}_2(s_{**})=\bar{u}_1(s_{**})<0$.
Thus by $\bar{u}_2(s^{*})=u_2^{+}(s^{*})>0$ we know that there exists a $s_{s}\in (s^{*}, s_{**})$ such that $\bar{u}_2(s_{s})=0$. Hence, the problem (\ref{AE}), (\ref{IBV}) admits a discontinuous solution with two compression shocks. The solution has the form
$$
(u, \rho)(x, t)=\left\{
                 \begin{array}{ll}
                   (u_1, \rho_1)(s), & \hbox{$s<s^{*}$,} \\[4pt]
                   (\bar{u}_1, \bar{\rho}_1)(s), & \hbox{$s^{*}<s<s_{s}$,}\\[4pt]
(0, \bar{\rho}_2(s_{s})), & \hbox{$s>s_{s}$;}
                 \end{array}
               \right.
$$
see Figure \ref{Figu2}.

\vskip 4pt

We next discuss the case for $\tau_1(s_*)\in \{\tau_2^i\}\cup(0, \tau_1^i]$ or $\tau_0>\tau_3$. If $F(s)\geq 0$ as $\tau_2^i<\tau_1(s)\leq\tau_3$, then $${u}_2(s):=\left\{
                                                                \begin{array}{ll}
                                                                 u_2^{+}(s), & \hbox{$F(s)=0$;} \\[4pt]
u_2(s), & \hbox{otherwise}
                                                                \end{array}
                                                              \right.
$$ is a continuous function on $(0, s_*]$. Moreover, ${u}_2(s_*)=u_1(s_*)<0$ and $\lim\limits_{s\rightarrow 0}{u}_2(s)=+\infty$. Hence, there exists a $s_{s}\in (0, s^{*})$ such that ${u}_2(s_{s})=0$, and consequently the problem (\ref{AE}), (\ref{IBV}) admits a discontinuous solution with a single compression shock; see Figure \ref{fig6}(right).
If $F(s)$ is not nonnegative as $\tau_2^i<\tau_1(s)\leq\tau_3$, then the discussion will be similar to the previous discussions.

\subsection{Equation of state III}
\subsubsection{$\tau_0\leq \tilde{\tau}_1$} The discussion is similar to that of section 4.1, since  $\tau_1'(s)<0$ as $s>0$ and $p''(\tau)>0$ as $\tau<\tau_0$.
\subsubsection{$\tilde{\tau}_1<\tau_0\leq \tilde{\tau}_2$}
Let $s_{*}$ be determined by
$$
\int_{\rho_0}^{\tilde{\rho}_1}\frac{1}{\rho}{\rm d}\rho~=~\int_{0}^{s_{*}}\frac{2u_0}{u_0 s-1} {\rm d}s.
$$
Hence, we have
$$
u_1(s)=u_0, \quad \rho_1(s)=\rho_0\exp\Big(\int_{0}^{s}\frac{2u_0}{u_0 s-1} {\rm d}s\Big), \quad 0<s<s_*.
$$
We then have the following two cases:
(1) $u_0< 1/s_{*}-b_1$;
(2) $u_0\geq 1/s_{*}-b_1$.

\begin{figure}[htbp]
\begin{center}
\includegraphics[scale=0.43]{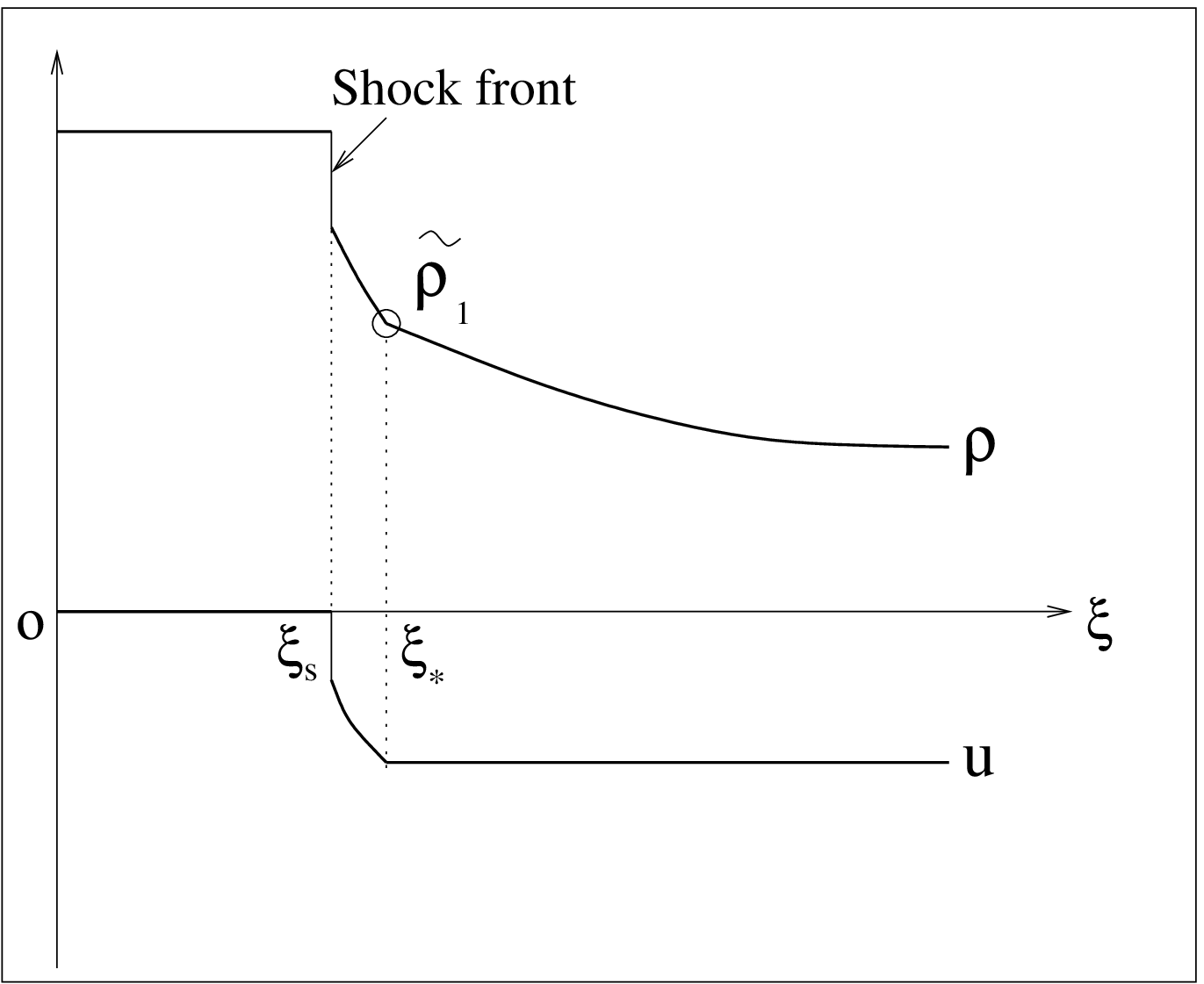}\quad\includegraphics[scale=0.42]{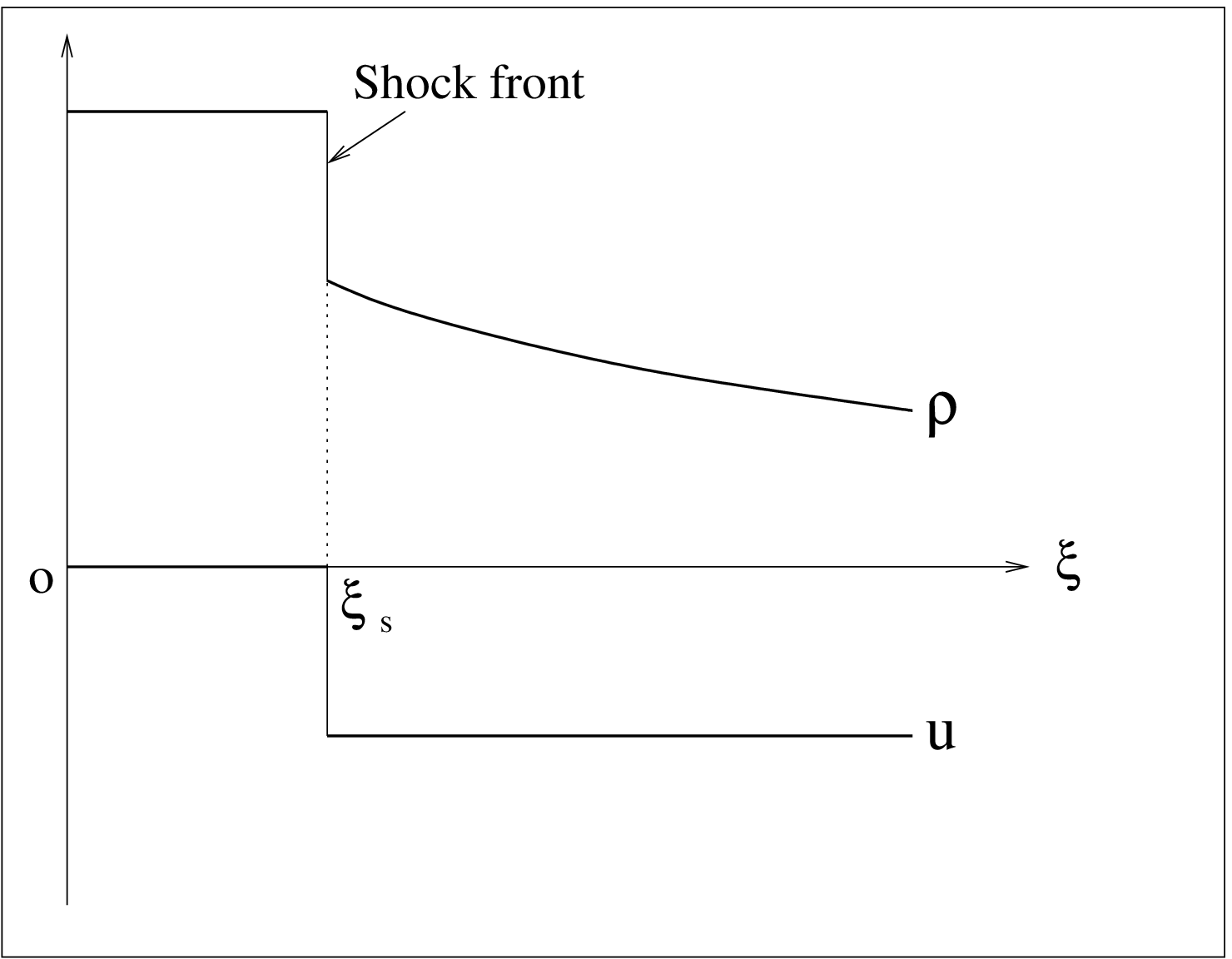}
\caption{\footnotesize Solutions with a single compression shock. Left: $u_0< 1/s_{*}-b_1$; right: $u_0\geq 1/s_{*}-b_1$.}
\label{fig7}
\end{center}
\end{figure}
If $u_0< 1/s_{*}-b_1$ then
we consider the system (\ref{ODE2}) with the initial data
\begin{equation}\label{4.2}
(u, \rho)(s_*)~=~(u_0, \tilde{\rho}_1).
\end{equation}
It is similar to Lemma \ref{lem4.1} that there exists a $s^{*}>s_*$
such that the solution $(u_1, \rho_1)(s)$ of the problem (\ref{ODE2}), (\ref{4.2}) satisfies $u_1(s)<h(\rho_1(s), s)$ as $s_*<s<s^*$ and $u_1(s^*)=h(\rho_1(s^*), s^{*})<0$.
Moreover, for any $s\in (0, s^{*})$
there exists an admissible forward compression shock with the speed $1/s$
and the front side state $(u_1, \rho_1)(s)$. The back side state of the shock $(u_2, \rho_2)(s)$ can be determined by (\ref{backside}). It is easy to see that
$u_2(s^{*})=u_1(s^*)<0$ and $\lim\limits_{s\rightarrow 0}u_2(s)=+\infty$.
Hence, there exists a $s_{s}\in (0, s^{*})$ such that $u_2(s_{s})=0$, and consequently the problem (\ref{AE}), (\ref{IBV}) admits a discontinuous solution with a single compression shock; see Figure \ref{fig7}.

Next, we discuss the case of $u_0\geq 1/s_{*}-b_1$. By Corollary \ref{cor5}, we know that
for any $s\in (0, s^{*})$,
there exists an admissible forward compression shock with the speed $1/s$
and the front side state $(u_1, \rho_1)(s)$. The back side state of the shock $(u_2, \rho_2)(s)$ can be determined by (\ref{backside}).
It is obviously that
$\lim\limits_{s\rightarrow 0}u_2(s)\rightarrow +\infty$.
Using $u_0> 1/s_{*}-b_1$, we also have $\lim\limits_{s\rightarrow s_*}\tau_2(s)=\tilde{\tau}_1$ and  $\lim\limits_{s\rightarrow s_*}u_2(s)=u_0<0$.
Thus, there must exists a $s_{s}\in (0, s_*)$ such that $u_2(s_s)=0$. The solution for this case can be illustrated by Figure \ref{fig7}(right).

\subsubsection{$\tau_0> \tilde{\tau}_2$}
We have the following three cases:
\begin{itemize}
  \item There exists a $s_{*}>0$ such that $u_1(s)<h(\rho_1(s), s)$ as $s<s_*$ and $u_1(s_*)=h(\rho_1(s_*), s_*)<0$ and $\tau_1(s_*)\in [\tilde{\tau}_2, \tau_0)$.
  \item There exists a $s_{*}>0$ such that $u_1(s)<h(\rho_1(s), s)$ as $s<s_*$ and $\tau_1(s_*)=\tilde{\tau}_1$ and $u_1(s_*)\geq\frac{1}{s_*}-b_1$.
  \item There exists a $s_{*}>0$ such that $u_1(s)<h(\rho_1(s), s)$ as $s<s_*$ and $u_1(s_*)=h(\rho_1(s_*), s_*)<0$ and $\tau_1(s_*)<\tilde{\tau}_1$. (Remark: $u_1(s_*)<\frac{1}{s_*}-b_1$ in this case.)
\end{itemize}

We now discuss the first case. 
For $\tau_1>\tilde{\tau}_2$, we let $\kappa(\tau_1)$ be defined such that
\begin{equation}
\kappa(\tau_1)=\frac{p(\tilde{\tau}_{2})-p(\tau_1)}{\tilde{\tau}_{2}-\tau_1}.
\end{equation}
Then we have $-\kappa(\tau_1)>-p'(\tau_1)$, since $\tau_1>\tilde{\tau}_2$.

Let
$$\hat{\xi}(s)=u_1(s)+\tau_1(s)\sqrt{-\kappa(\tau_1(s))}, \quad F(s)=\frac{1}{s}-\hat{\xi}(s).$$
Then we have
\begin{equation}\label{12202}
\lim\limits_{s\rightarrow 0}F(s)=+\infty
\end{equation}
 and
\begin{equation}\label{11501}
\begin{aligned}
F(s_*)&=\frac{1}{s_*}-u_1(s_*)-\tau_1(s_*)\sqrt{-\kappa(\tau_1(s_*))}\\&=
\frac{1}{s_*}-u_1(s_*)-\tau_1(s_*)\sqrt{-p'(\tau_1(s_*))}+\tau_1(s_*)\Big(\sqrt{-p'(\tau_1(s_*))}-\sqrt{-\kappa(\tau_1(s_*))}\Big)
\\&=\tau_1(s_*)\Big(\sqrt{-p'(\tau_1(s_*))}-\sqrt{-\kappa(\tau_1(s_*))}\Big)<0.
\end{aligned}
\end{equation}
Since $
1/s~>~u_1(s)+\sqrt{p'(\rho_1(s))}
$ and $\tilde{\tau}_2<\tau_1(s)<\tau_0$
as $0<s<s_*$, for any $s\in (0, s_*)$,  there exists an admissible compression shock with the speed $1/s$ and the front side state $(u_1, \rho_1)(s)$. The back side state of the shock $(u_2, \rho_2)(s)$ ($0<s\leq s_*$) can be determined by (\ref{backside}). Moreover, we have
$u_2(s_{*})=u_1(s_*)<0$ and $\lim\limits_{s\rightarrow 0}u_2(s)=+\infty$.
However, by (\ref{12202}) and (\ref{11501}) we know that $u_2(s)$ is not continuous in $(0, s_*)$.
Since, if $F(s)=0$ then (\ref{backside}) has two solutions $(u_{2}^{+}, \rho_{2}^{+})(s)$ and $(u_{2}^{-}, \rho_{2}^{-})(s)$, where $u_{2}^{+}(s)>u_{2}^{-}(s)$ and $\rho_{2}^{+}(s)>\rho_{2}^{-}(s)$.

If there exists a $s_{s}\in (0, s_*)$ such that $u_2(s_s)=0$, then the problem (\ref{AE}), (\ref{IBV}) admits a discontinuous solution with a single shock.

If $u_2(s)\neq0$ for all $s\in(0, s_*)$, then there must exists a $s^{*}\in (0, s_*)$ such that
$u_2^{+}(s^{*})>0$, $u_2^{-}(s^{*})<0$, and
$\tau_2^{-}(s^{*})=\tilde{\tau}_2$.
 Then the discussion for $s>s^{*}$ is similar to that of section 4.3.2.
The problem has a discontinuous solution with two compression shocks; see Figure \ref{fig10}.

\begin{figure}[htbp]
\begin{center}
\includegraphics[scale=0.43]{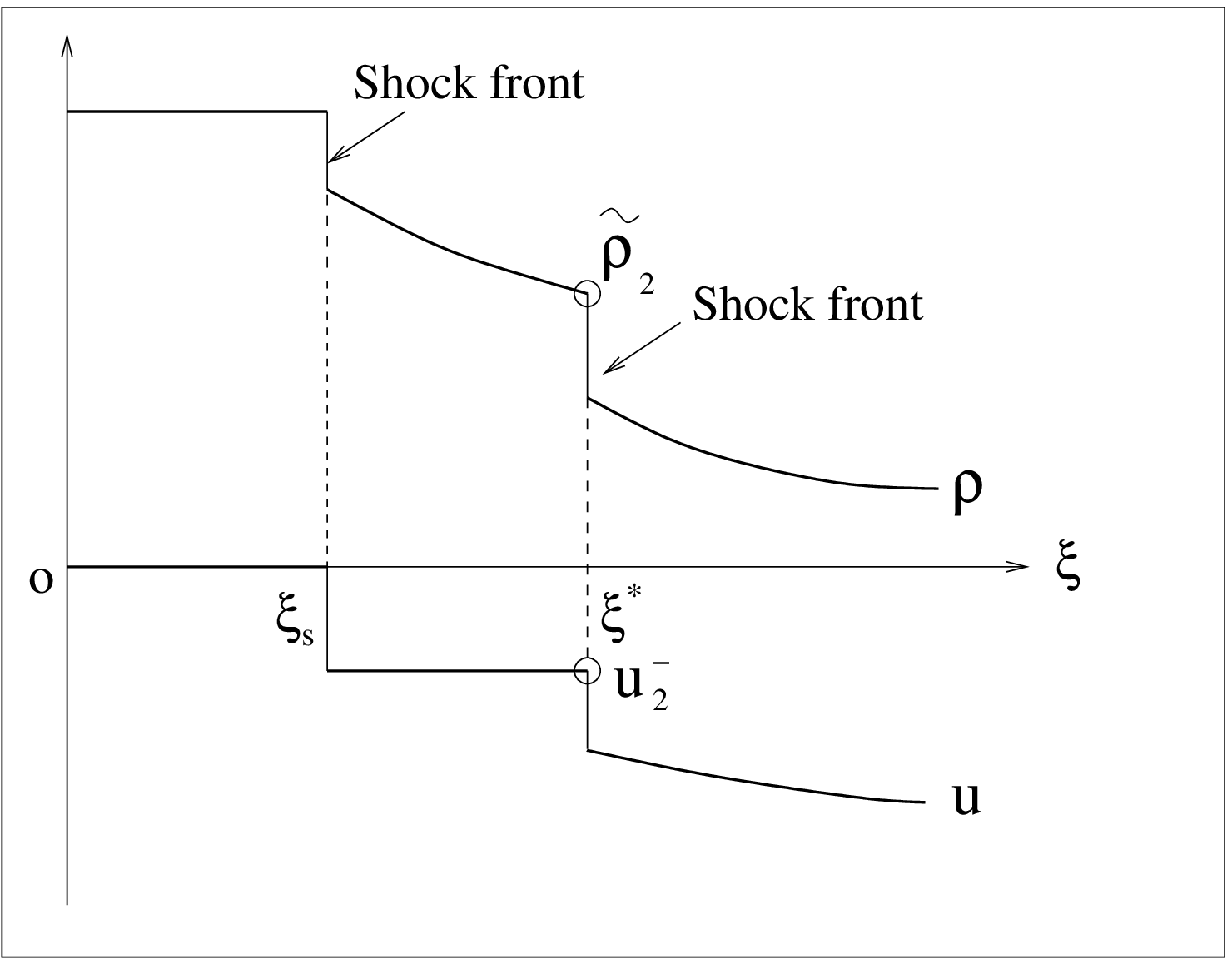}\quad \includegraphics[scale=0.43]{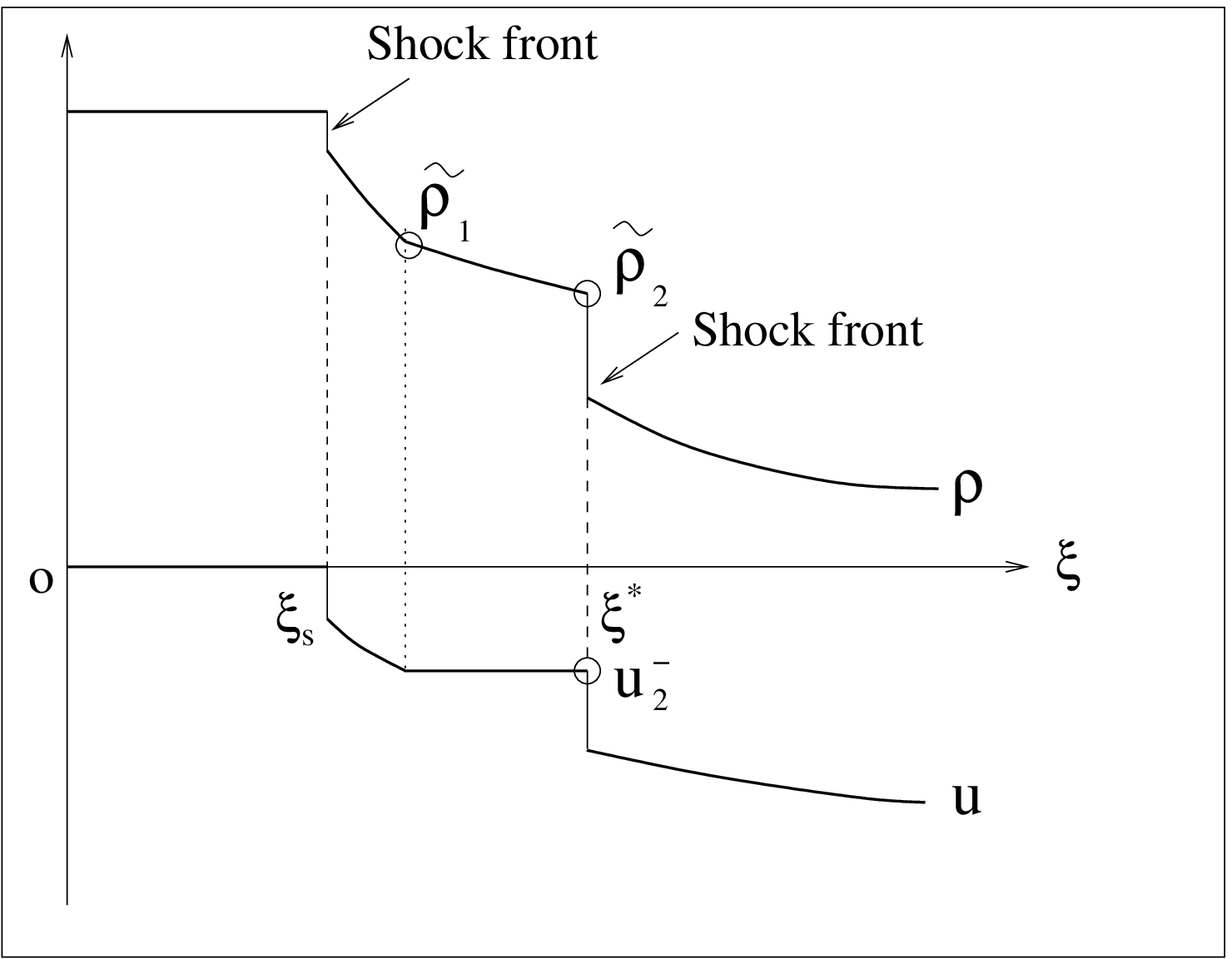}
\caption{\footnotesize Solutions with two compression shocks for $u_0<0$ and $\tau_0>\tilde{\tau}_2$.}
\label{fig10}
\end{center}
\end{figure}

We next discuss the second case.
Similarly, for any $s\in (0, s_*)$,  there exists an admissible compression shock with the speed $1/s$ and the front side state $(u_1, \rho_1)(s)$. The back side state of the shock $(u_2, \rho_2)(s)$ ($0<s\leq s_*$) can be determined by (\ref{backside}).
Moreover, by $u_1(s_*)\geq\frac{1}{s_*}-b_1$ we have $\lim\limits_{s\rightarrow s_*}u_2(s)=u_1(s_*)<0$.
Let $s_1$ be the point such that $\tau_1(s_1)=\tilde{\tau}_2$.
Then the discussion can be divided into the following two cases: (1) $F(s)\geq 0$ as $s\in (0, s_1)$; (2) $F(s)$ is not nonnegative in $(0, s_1)$.

If $F(s)\geq 0$ as $s\in (0, s_1)$, we redefine $u_2(s)=\left\{
                                                                \begin{array}{ll}
                                                                  u_2(s), & \hbox{$F(s)>0$;} \\
                                                                  u_2^{+}(s), & \hbox{$F(s)=0$}
                                                                \end{array}
                                                              \right.$ as $0<s<s_1$. Then $u_2(s)$ is a continuous function on $(0, s_*)$.  Thus, there exists a $s_{s}\in (0, s_{*})$ such that $u_2(s)=0$. And consequently, the problem (\ref{AE}), (\ref{IBV}) admits a discontinuous solution with a single compressible shock.
If $F(s)$ is not nonnegative on $(0, s_1)$ then the discussion will be similar to the first case.

Actually, the discussion for the third case is similar to that of the second case. We omit the details.

\vskip 32pt
\centerline{\sc \bf Acknowledgements}%
\vskip 8pt
The first author is partially supported by Natural Science foundation of Zhejiang Province (LQ19A010003).
The second author is partially supported by the grant of ``Shanghai Altitude
Discipline".

\vskip 32pt

\end{document}